\documentclass[11pt]{amsart}
\usepackage{amsmath,amsthm, amsfonts,latexsym,amssymb, verbatim}
\usepackage{url}
 \usepackage[OT2, T1]{fontenc}

\usepackage{pbox}
\usepackage{graphicx,caption}
\usepackage{amscd}
\usepackage{caption,booktabs}
\usepackage{amsrefs} 

\usepackage{setspace}
\usepackage[flushleft]{threeparttable}

\usepackage[noabbrev]{cleveref}

\newtheorem{thm}{Theorem}
\newtheorem{prop}[thm]{Proposition}
\newtheorem{lem}[thm]{Lemma}

\theoremstyle{definition}

\newcommand{\tors}{\mathrm{tors}}

\begin{document}

\title[]{ Growth of Torsion of Elliptic Curves with Full 2-torsion over Quadratic Cyclotomic Fields
}
\author{Burton Newman}
\address{Department of Computer Science\\ University of Southern California\\ Los Angeles, CA 90089\\ USA}
\email{bnewman@usc.edu}
\thanks{The author was partially supported by the NSF grant DMS-1302886.}

\keywords{Elliptic curves, torsion subgroups, quadratic twists}
\subjclass[2010]{Primary: 11G05, Secondary: 14G05}
\begin{abstract}
Let $K = \mathbb{Q}(\sqrt{-3})$ or $\mathbb{Q}(\sqrt{-1})$ and let $C_n$ denote the cyclic group of order n.  We study how the torsion part of an elliptic curve over $K$ grows in a quadratic extension of $K$.  
In the case $E(K)[2] \approx C_2 \oplus C_2$ we determine how a given torsion structure can grow in a quadratic extension and the maximum number of quadratic extensions in which it grows.  We also classify the torsion structures which occur as the quadratic twist of a given torsion structure.  

\end{abstract}
\maketitle

\section{Introduction}

Let $K$ be a number field.  If $E/K$ is an elliptic curve, let $E(K)_{tor}$ denote the torsion part of the group $E(K)$. 
Much research has been devoted to 
determining the finite abelian groups (up to isomorphism) which occur as 
$E(K)_{tor}$ for some elliptic curve $E/K$.  
The classification of the possible torsion structures
when $K=\mathbb{Q}$ was first completed by Mazur. 
Let $C_n$ denote the cyclic group of order $n$.

\begin{thm}{(Mazur \cite{Mazur78}, Kubert \cite{Kubert76})}\label{Mazur}
Let E be an elliptic curve over $\mathbb{Q}$.  Then the torsion subgroup $E(\mathbb{Q})_{tor}$ of $E(\mathbb{Q})$ is isomorphic to one of the following 15 groups:
\vspace{5mm}
\begin{itemize}
\item $C_n$ for $1 \leq n \leq 12, n \ne 11$\
\item $C_2 \oplus C_{2n}$ for $1 \leq n \leq 4$\\
\end{itemize}
\noindent Furthermore, each group above occurs infinitely often (up to isomorphism).
\end{thm}

\noindent 

By contrast, when $K=\mathbb{Q}(\sqrt{5})$, there is only one elliptic curve over $K$ (up to isomorphism) with torsion part $C_{15}$ \cite{NajCubic}.
The work of Mazur was generalized by Kamienny et al. to the case of quadratic number fields.  

\begin{thm}{(Kamienny \cite{Kamienny}, Kenku and Momose \cite{Kenku&Momose}}) \label{Kamienny}
Let K be a quadratic field and E an elliptic curve over K.  Then $E(K)_{tor}$ is isomorphic to one of the following 26 groups:

\begin{itemize}

\item $C_n$ for $1 \leq n \leq 18, n \ne 17$\
\item $C_2 \oplus C_{2n}$ for $1 \leq n \leq 6$\
\item $C_3 \oplus C_{3n}$ for  $1 \leq n \leq 2$\ 
\item $C_4 \oplus C_{4}$\

\end{itemize}

\end{thm}

\noindent Note that the above result classifies the possible torsion structures as one varies over all quadratic number fields.  Hence for a given quadratic field $K$, it is in general still unknown which torsion structures occur over $K$ among elliptic curves $E/K$.  Two cases in which the answer is known are $K = \mathbb{Q}(\sqrt{d})$, $d=-1,-3$:

\begin{thm}({Najman \cite{NajCyc}})
\label{Najman}
Let K be a cyclotomic quadratic field and E an elliptic curve over K.  
\begin{itemize}

\item If  $K = \mathbb{Q}(i)$ then $E(K)_{tor}$ is either one of the groups from Mazur's theorem or $C_4 \oplus C_{4}$.
 
\item  If  $K = \mathbb{Q}(\sqrt{-3})$ then $E(K)_{tor}$ is either one of the groups from Mazur's theorem, $C_3 \oplus C_{3}$ or $C_3 \oplus C_{6}$.

\end{itemize}
\end{thm}

\noindent For cubic and quartic number fields, we still do not have a complete classification analogous to  \cref{Kamienny}. We do know the torsion structures that occur infinitely often (up to isomorphism) \cites{Jeon04,Jeon06}.  We also know the possible prime divisors of the order of the torsion over extensions of degree $n \leq 6$
\cites{Parent00,Parent03,Stoll} (and unpublished work of Derickx, Kamienny, Stein and Stoll).
Furthermore, for cubic fields, we have a classification of possible torsion structures if we restrict to elliptic curves defined over $\mathbb{Q}$ \cite{NajCubic}. We also have a classification of torsion structures over the compositum of all quadratic fields for elliptic curves defined over $\mathbb{Q}$ \cite{Fujita} as well as significant progress over the compositum of all cubic fields \cite{Daniels}. \cref{Kamienny} provided evidence  for the Strong Uniform Boundedness conjecture, which was confirmed by Merel:

\begin{thm}({Merel \cite{Merel}}) \label{Merel}
For every positive integer d there exists B$_d > 0$ such that for every
number field K with $[K:\mathbb{Q}] = d$  and every elliptic curve E/K we have 

$$|E(K)_{tor}| \leq B_d$$

\end{thm}

With the above classification results in hand, researchers began 
to study how the torsion part of an elliptic curve grows in a quadratic extension of the base field.  
As we will describe in Section 2, if $L$ is a quadratic extension of $K$ and $E/K$ is an elliptic curve, there is a bound on $|E(L)_{tor}|$ which depends only on torsion structures that occur over $K$.  Furthermore, as a consequence of \cref{Merel} and \cref{torsionoverc} one can show that for any number field $K$, E(K)$_{tor}$ grows in a finite number of quadratic extensions  \cite{NajLarge}*{Thm 12}.

The case of quadratic growth when $K=\mathbb{Q}$ has already been studied by various authors.  Previously Kwon completed the case in which $E(\mathbb{Q})[2] = C_2 \oplus C_{2}$ \cite{Kwon}. 
In order to classify growth, Kwon studied the the distribution of torsion structures among the quadratic twists of an elliptic curve. The \textit{d-quadratic twist} of $E: y^2=x^3+ax+b$ is the elliptic curve with model $E^d: dy^2=x^3+ax+b$.
Kwon's result was extended by Tornero et al. to complete the classification over $\mathbb{Q}$  \cite{Tornero}.  On the other hand, they obtained non-optimal upper bounds on the number of extensions in which a given torsion structure can grow.  Najman finally determined the best possible bounds for each torsion structure over $\mathbb{Q}$ \cite{NajLarge}. In particular, no curve over $\mathbb{Q}$ can grow in more than 4 quadratic extensions.

Our research is a first step toward classifying growth of torsion in quadratic extensions of base fields larger than $\mathbb{Q}$.
We choose to work over the base fields $\mathbb{Q}(\sqrt{-3})$ and $\mathbb{Q}(\sqrt{-1})$, as these are the only two number fields besides $\mathbb{Q}$ for which a classification of torsion structures is known.
Let $K = \mathbb{Q}(\sqrt{-3})$ or $\mathbb{Q}(\sqrt{-1})$.
In the case $E(K)[2]=C_2 \oplus C_2$, we generalize Kwon's work, determining (i) a classification of the torsion structures  which occur as the quadratic twists of a given torsion structure (\cref{twist})
(ii) a classification of the torsion structures which occur as the growths of a given torsion structure (\cref{main-1}, \cref{main-3})
(iii) Tight bounds on the number of quadratic extensions in which a given torsion structure can grow (\cref{main-1}, \cref{main-3}).
Kwon's methods relied upon the fact that $\mathbb{Z}$ was a UFD and had  finite unit group.  Our methods don't depend explicitly on these properties and hence generalize more easily to arbitrary number fields. 
In order to seek out examples of growth (vi) we used division polynomials to write a program in Magma which on input an elliptic curve $E/K$, outputs a list of quadratic extensions of K where the torsion grows and the torsion in each case.

In order to describe our results, we introduce the following notation.
As in \cite{Tornero},
let $\Phi_K(d)$ be the set of possible groups that can appear as the torsion subgroup of an elliptic curve $E/K$ over a degree $d$ extension of $K$. 
Given $G \in \Phi_K(1)$, write $\Phi_K(d,G)$ for the set of possible groups that can appear as the torsion subgroup over a degree $d$ extension $L/K$
, of an elliptic curve $E/K$ with $E(K)_{\tors} \approx G$.
Given $G \in \Phi_K(1)$, let $T_K(G)$ denote set of finite abelian groups H (up to isomorphism) such that there is an elliptic curve $E/K$ with $E(K)_{tor} \approx G$ and $E^d(K)_{tor} \approx H$ for some $d \in K$.  
If $E/K$ is an elliptic curve,  let $g(E)$ denote the number of quadratic extensions $L$ of $K$ such that $E(L)_{tor} \not = E(K)_{tor}$.  If $G$ is a finite abelian group, let $g(G)$ denote the set of positive integers $m$ such that $g(E) = m$ for some elliptic curve $E/K$ with $E(K)_{tor} \approx G$. 
A summary of our main results  appear in \cref{tableofresults} and \cref{tableofresults2}.  For a more detailed description of growth of torsion see \cref{main-1} and \cref{main-3} and for a description of quadratic twists see \cref{twist}.

\begin{table}

 \caption{ Summary of Main Results: Growth of Torsion}

\begin{center}

\resizebox{\linewidth}{!}{
\begin{tabular}{ |c|c|c|c|c| } 
 \hline
  \label{tableofresults}

  & \multicolumn{2}{|c|}{$K=\mathbb{Q}(\sqrt{-3})$} & \multicolumn{2}{|c|}{$K=\mathbb{Q}(\sqrt{-1})$} \\
 \hline 
 
 $G$ & $\Phi_K(2,G) \setminus G$  & $ g_K(G)$ & $\Phi_K(2,G) \setminus G$ &  $g_K(G)$ \\ 
 \hline

 $C_2 \oplus C_8$  & $\emptyset$  & 0 & $C_4 \oplus C_8$  & 0,1 \\ 
 \hline
 
 $C_2 \oplus C_6$  & $C_2 \oplus C_{12}$  & 0,1 &  $C_2 \oplus C_{12}$&0,2 \\ 
 \hline
 
 $C_2 \oplus C_4$  & $C_4 \oplus C_4$,$C_2 \oplus C_8$  & 0,1,2,3  & $C_4 \oplus C_4$,$C_2 \oplus C_8$   & 0,1,2,3 \\ 
 \hline
 
 $C_4 \oplus C_4$  & --- & ---& $C_4 \oplus C_8$  &   0,1\\ 
 \hline
  $C_2 \oplus C_2$ 
  
  &\pbox{20cm}{ $C_2 \oplus C_{2n}$\\($n=2,3,4,6$) \\ $C_4 \oplus C_4$\\ }

  & 0,1,2,3  
 
  & \pbox{20cm}{ \  \\ $C_2 \oplus C_{2n}$\\($n=2,3,4,6$) \\ $C_4 \oplus C_4$\\ $C_4 \oplus C_8$\\}
  & 0,1,2,3\\ 
 \hline
\end{tabular}

}
\end{center}
\end{table}


\begin{table}

 \caption{ Summary of Main Results: Quadratic Twists over $K=\mathbb{Q}(\sqrt{d})$}

\begin{center}

\begin{tabular}{ |c|c|c| } 
 \hline
  \label{tableofresults2}

  & \multicolumn{2}{|c|}{$T_K(G)$} \\
 \hline

 G & $d=-3$  & $d=-1$  \\
 \hline

 $C_2 \oplus C_8$  &  $C_2 \oplus C_2$    &  $C_2 \oplus C_2$    \\ 
 \hline
 
 $C_2 \oplus C_6$  & $C_2 \oplus C_2$    &  $C_2 \oplus C_2$ \\ 
 \hline
 
 $C_2 \oplus C_4$  & $C_2 \oplus C_2$,$C_2 \oplus C_4$  &    $C_2 \oplus C_2$     \\ 
 \hline
 
 $C_4 \oplus C_4$  & --- &  $C_2 \oplus C_2$    \\ 
 \hline
  $C_2 \oplus C_2$ 
  
  &\pbox{20cm}{ $C_2 \oplus C_{2n}$\\($n=1,2,3,4$)  }

  & \pbox{20cm}{ \  \\ $C_2 \oplus C_{2n}$\\($n=1,2,3,4$) \\ $C_4 \oplus C_4$\\}
   \\ 
 \hline
\end{tabular}

\end{center}
\end{table}

Now we give an overview of  the paper.
In \cref{Background} we introduce quadratic twists, which play a crucial role in the study of growth.
If  $L = K(\sqrt{d})$, $E/K$ is an elliptic curve and $E(K)_{tor}$ is given then  $E^d(K)_{tor}$ 
yields a bound for the 2-part of $E(L)_{tor}$ 
(see \cref{kwon}). 
Hence if one wants to study the growth of the torsion structure $E(K)_{tor}$ the first task is to classify which torsion structures occur as $E^d(K)_{tor}$ (see \cref{Classification of Twists}).
The classification of twists uses a parametrization of elliptic curves containing a given torsion structure $C_2 \oplus C_{2n}$ ($1 \leq n \leq 4$) (\cref{Ono}).  Explicitly twisting and equating parametrized models of curves 
yields systems of diophantine equations.  These systems are solved in 
\cref{Some Diophantine}. 

Now, given a torsion structure  $E(K)_{tor}$, the classification of twists yields a finite list of possible torsion structures for the 2-part of $E(L)$.  Furthermore, the 2-part of $E(K)$ grows in $L$ if and only if for some point  $P \in E(K)_{tor}$, $P \in 2E(L)\setminus 2E(K)$. Given $E/K$ such that $E(K)[2] = C_2 \oplus C_2$ and $P \in E(K)$, \cref{lifting lemma} provides a criteria for when $P \in 2E(L)$.
 Applying this criteria to parametrized models 
as in \cref{Ono} yields systems of Diophantine equations.

Finding all the $K$-rational solutions to the Diophantine equations arising in our research would have been quite difficult without computer assistance.
We used the software packages Magma and Sage.   Given a curve $C/K$, Magma can determine the genus of $C$.  If the genus is 1 and we can find just one $K$-rational point $P$ on $C$, then Magma can find a Weierstrass model $E$ for $C$ and a birational map $f: C \rightarrow E$. 
 Given the Weierstrass equation for $E$, Magma can now try to compute the rank of $E(K)$.  If the rank is 0, then $E(K)=E(K)_{tor}$ is finite and Magma  can carry out standard algorithms to find all of $E(K)$. Finally one can use the map $f$ to recover all of $C(K)$.
    
If the curve  $C$ has genus at least 2, then $C(K)$ does not form a group as above, but the $K$-points of its Jacobian $J(C)(K)$ do.  
We have an injective morphism $f:C \rightarrow J(C)$. If $C$ is genus 2 hyperelliptic, Magma has functionality analogous to the the tools above for elliptic curves:  Magma can determine if $J(C)(K)$ is finite and can help determine the torsion subgroup of $J(C)(K)$.  
Finally one can often use the morphism $f$ to recover $C(K)$.

The algorithm we wrote to find examples of growth of torsion also relied upon Magma 
to factor division polynomials over quadratic fields, as well as compute the torsion over quartic number fields.  

Many of the calculations in our paper follow the ideas of Kwon \cite{Kwon}.  In a future paper we will describe progress on a similar classification of quadratic growth over quadratic cyclotomic fields in the case
$E(K)[2] = C_1$.

\section{Background}
\label{Background}
If $E$ is an elliptic curve with model
$$E: y^2 = x^3+Ax+B$$  Then its \textit{d-quadratic twist} is  $$E^d: y^2 = x^3+d^2Ax+d^3B$$  Note this curve is isomorphic over K to the curve with model $$dy^2 = x^3+Ax+B$$ so we often use them interchangeably.  On the other hand the map 
$$ T: E^d \rightarrow E$$ 
$$ (x,y) \mapsto (x,y\sqrt{d})$$
defines an isomorphism over $K(\sqrt{d})$.  Hence if $d$ is a square in $K$ then $E$ and $E^d$ are isomorphic over all extensions of $K$ and if $d$ is not a square then we have  $E(K(\sqrt{d})) \approx E^d(K(\sqrt{d}))$.

We also have the following proposition of Kwon, which allows us to bound the growth  $E(K)_{tor}$ in a quadratic extension.

\begin{prop}\cite{Kwon}\label{kwon}
Let K be a quadratic number field, $L=K(\sqrt{d})$ a quadratic extension of K and let $\sigma$ denote the generator of Gal(L/K).
The map $$ E(L)_{tor} \rightarrow E^d(K)_{tor}$$
   $$P \mapsto P - \sigma(P)$$

is a homomorphism with kernel $E(K)_{tor}$ and hence induces an injection 
$$ E(L)_{tor}/E(K)_{tor} \hookrightarrow E^d(K)_{tor}$$

\end{prop}

\begin{prop}\cite{Tornero}*{Cor. 4}
\label{oddpartsadd}
If n is an odd positive integer we have 
$$E(K(\sqrt{d}))[n] \approx E(K)[n] \oplus E^d(K)[n]$$ 
\end{prop}

\begin{thm}\cite{Si}*{Cor. 6.4}
\label{torsionoverc}
If K is a field with char(K)=0, E/K is an elliptic curve and n is a positive integer then  
$$E(\overline{K})[n] \approx C_n \oplus C_n$$
\end{thm}

The following lemma plays a crucial role in studying the 2-part of the torsion of an elliptic curve:
\begin{lem}\cite{Knapp}
\label{lifting lemma}
Let E be an elliptic curve over a field F with char(F) $\ne 2,3$.  Suppose E is given by 
$$ y^2 = (x-\alpha)(x-\beta)(x-\gamma)$$
with $\alpha, \beta, \gamma \in F$.  For $P=(x_0,y_0)$ in E(F) there exists Q in E(F) such that $2Q=P$ iff    $x_0-\alpha, x_0-\beta, x_0-\gamma$ are squares in F.

\end{lem}
If a point $P$ meets the criteria in the theorem above, we will say the point $P$ $\textit{lifts}$ in $K$.  Below we have a variant of a result of Ono.  His result held for elliptic curves over $\mathbb{Q}$.  Note the result below does not depend on the class group nor the unit group of the number field $K$.
\begin{thm}\cite{Ono}
\label{Ono}
Let K be  number field and E/K an elliptic curve with full 2-torsion. Then E has a model of the form $y^2=x(x+M)(x+N)$ where M,N $\in \mathcal{O}_K$.
\begin{enumerate}
\item  E(K) has a point of order 4 iff M, N are both squares (in $\mathcal{O}_K$), or -M, N-M are both squares, or -N,M-N are both squares.
\item   E(K) has a point of order 8 iff  there exists a d $\in \mathcal{O}_K$, $d \ne 0$ and a pythagorean triple (u,v,w) (i.e. u,v,w $\in \mathcal{O}_K$ with $u^2+w^2=v^2$)  such that $ M = d^2u^4, N=d^2v^4$ or  $-M = d^2u^4, N-M=d^2v^4$, or $ -N = d^2u^4, M-N=d^2v^4$.
\item  E(K) contains a point of order 3 iff there exists $a,b,d \in \mathcal{O}_K$ with $d \ne 0$ and $\frac{a}{b} \not\in \{-2,-1,-\frac{1}{2},0,1\}$ such that $M=a^3(a+2b)d^2$ and $N=b^3(b+2a)d^2$.

\end{enumerate}
\end{thm}

\begin{proof}

1) E has a point of order 4 over K iff (0,0), (-M,0) or (-N,0) lifts over K.  But by the     \cref{lifting lemma}, (0,0) lifts iff M,N are squares in K, (-M,0) lifts iff $-M, N-M$ are squares in K and (-N,0) lifts iff $-N, M-N$ are squares in K.  Also note that since  $\mathcal{O}_K$ is integrally closed, if $x \in  \mathcal{O}_K$ is a square in K then it is a square in  $\mathcal{O}_K$.

2)  Suppose E(K) has a point P of order 8.  By translating if necessary, we may assume WLOG that $4P = (0,0)$.  Hence by  \cref{lifting lemma} there exist $a,b \in  \mathcal{O}_K$ such that $M=a^2$ and $N=b^2$ and the four order 4 points above (0,0) are:
$$ (ab, \pm ab(a+b))    $$
$$    (-ab, \pm ab(a-b)) $$
If $[2]:E \rightarrow E$ denotes the multiplication-by-2 map on E, note that $|Ker([2])| = 4$ so we expect 4 points above.  Replacing $a$ by $-a$ if necessary, we may assume WLOG that $ (ab, \pm ab(a-b))$ lifts over K. By  \cref{lifting lemma} $ab, ab+a^2$ and $ab+b^2$ are squares in K so there exists $u,w \in 
\mathcal{O}_K$ such that $ab=u^2$ and $ab+b^2=w^2$. Hence if we let $d = b^{-1}$ then:
$$M=a^2= (b^{-1})^2(ab)^2=d^2u^4$$
$$N = b^2 = d^2b^4$$ Furthermore, $$u^2+b^2 = ab + b^2 = w^2$$
The converse is easy to verify by \cref{lifting lemma}.\\

3)  A point $(x,y)$ on our curve is order 3 iff $x$ is a root of the 3-division polynomial:
$$3x^4+4(M + N)x^3 + 6MNx^2 - M^2N^2 = 0$$
As Ono notes, this curve (in M,N,x) has a rational parametrization (due to Nigel Boston) of the form:
$$\frac{M}{x} = (1+t)^2-1 $$
$$ \frac{N}{x} =  (1+t^{-1})^2-1$$
Hence $$M = x((1+t)^2-1)$$
$$N = x((1+t^{-1})^2-1)$$ and as $(x,y)$ lies on our curve, we have have:
\begin{align*}
y^2 & = x(x+M)(x+N)\\
& = x(x+ x((1+t)^2-1))(x+x((1+t^{-1})^2-1) )\\
& =x^3(1+(1+t)^2-1)(1+(1+t^{-1})^2-1)\\
& = x^3(1+t)^2(1+t^{-1})^2\\
\end{align*}
and hence $x$ is a square:
$$x = [\frac{y}{x(1+t)(1+t^{-1})}]^2$$
\noindent Let $c = \frac{y}{x(1+t)(1+t^{-1})}$ so we have $x=c^2$. Then:
\begin{align*}
M & = c^2((1+t)^2-1) \\
& = c^2(t^2+2t)\\
& =( \frac{c}{t})^2 t^3(t+2)\\
\end{align*}
Similarly:
\begin{align*}
N & = c^2((1+t^{-1})^2-1) \\
& = (\frac{c}{t})^2(2t+1)\\
\end{align*}

Writing $t := \frac{a}{b}$ for some a,b $\in$ $\mathcal{O}_K$ and  $d := \frac{c}{tb^2}$, the result follows. 

\end{proof}

\section{Some Diophantine Equations}

\label{Some Diophantine}
The curve $x^2+y^2=z^2$ is genus 0 and hence can be parametrized.  

\begin{prop}
\label{circle}

Let K be a number field and let C be the curve $$x^2+y^2 = z^2$$  For every K-rational projective point P $\ne [-1,0,1]$ on C there exists an $m \in K$ such that $P = [1-m^2, 2m, 1+m^2]$.

\end{prop}

\begin{prop}
\label{Fermat}

Let K= Q$(\sqrt{D}$) (with D a squarefree integer) be any imaginary quadratic number field with class number 1 except D=-7.  The equation 

$$u^4-v^4 = w^2$$

has no solutions $(u,v,w)$ with $u,v,w \in \mathcal{O}_K$ and $(u,v,w) \ne (0,0,0)$.
\end{prop}

\begin{proof}

Let C be the (affine) curve with equation:

$$u^4-v^4 = w^2$$

and let E be the elliptic curve:

$$y^2 = x^3+4x$$

The curves E and C are birational via:

$$f: C \to E$$ 
 $$(u,v) \mapsto (x,y) = \left( \frac{2v^2}{u^2-w}, \frac{4uv}{w-u^2} \right)$$

Hence, in order to determine the K-rational points on C, it suffices to consider the non-regular points of f and $f^{-1}(E(K))$.  At a non-regular point (u,v,w) we have $w=u^2$ and hence $v=0$.  Now using Magma we find the rank of $E(K)$ is 1 if $D=-7$ and 0 otherwise.  Also, $E(K)_{tor} \approx C_2 \oplus C_4$ if $D=-1$ and otherwise $E(K)_{tor} \approx C_4$.  Furthermore, the torsion points over K=Q(i) are $(0,0), (2,\pm 4), (\pm 2i,0), (-2,\pm 4i)$.  If y = 0 then we see u=0 or v=0.  If $x= \pm 2$ and $y=4u_0$ where $u_0$ is a unit in K, then it follows that either v=0 or v and u differ by a unit, in which case $w^2 = u^4-v^4 = 0$ so w = 0.  In more detail, if  $\frac{2v^2}{u^2-w} = \pm2$ and $\frac{4uv}{w-u^2} = 4u_0$ then $ \pm v^2 = u^2-w$ and $ u_1uv = u^2-w $ for some unit $u_1$ in K.  Hence $\pm v^2 = u_1uv$ so $\pm v(v-u_1u)=0$.  Hence $v=0$ or $v=u_1u$.  Therefore the curve C has no nontrivial solutions over K.  For D $\not = $ -1,-7 we have E(K)$= \{(0,0), (2,\pm 4)\}$ and the same argument shows these points do not yield nontrivial solutions over K.

\end{proof}
The extra points in the $D=-7$ case give rise to growth that doesn't occur in the $D=-1,-3$ cases. For example, the curve E defined over $K = \mathbb{Q}(\sqrt{-7})$ by:
$$y^2=x(x+u^4)(x+v^4)$$
where $u = \frac{1}{2}-\frac{3}{2}\sqrt{-7} $ and $v=-3- \sqrt{-7}$ has $E(K) \approx C_2 \oplus C_8$
but over $ L = K(\sqrt{\frac{465}{2}+\frac{45}{2}\sqrt{-7}})$  we have $E(L) \approx C_2 \oplus C_{16}$.

\begin{prop}
\label{order2=order3}
Let K = Q$(\sqrt{D}$) $(D=-1,-2,-3,-7,-11)$ and let s,t,d $\in \mathcal{O}_K$ with $s,t,d \ne 0$ and $s^2 \ne t^2$.  Then there do not exist  a,b $\in \mathcal{O}_K$ such that
$$ds^2=a^3(a+2b)c^2$$
$$dt^2 = b^3(b+2a)c^2$$
\end{prop}

\begin{proof}
It's easy to check $\frac{a}{b} \not\in \{-2,-1,-\frac{1}{2}, 0, 1\}$ as otherwise a hypothesis is contradicted.  Let K = Q$(\sqrt{D}$) (D=-1,-2,-3,-7,-11) and let s,t,d $\in K$ with s,t,d$ \ne 0$, $s^2 \ne t^2$ .  Suppose there exists a,c $\in K$ such that
$$ds^2=a^3(a+2b)c^2$$
$$dt^2 = b^3(b+2a)c^2$$ Taking products and dividing both sides by $a^2c^4b^4$ yields:

$$  (\frac{dst}{ab^2c^2})^2 =w(w+2)(1+2w) $$  where $w := \frac{a}{b}$.  Multiplying by 4 yields:  
$$  (\frac{2dst}{ab^2c^2})^2 =2w(2w+4)(2w+1) $$

Hence we a get a K-point on the elliptic curve E with equation

$$ y^2 = x(x+4)(x+1) = x^3+5x^2+4x $$

Magma tells us E(K) has rank 0 and Sage tells us E(K)$_{tor}$
$\approx C_2 \oplus C_4$.  
so E(K)$=E(K)_{tor} = \{ (-4,0), (-2 , \pm 2 ),  (-1 , 0), (0 ,0 ),  (2 , \pm 6),  \infty \} $.  Hence  $2w = -4, \pm 2, -1$ or $0$ so $w=-2,\pm 1$ or $0$, a contradiction.

\end{proof}

\begin{prop}
\label{modifiedorder2=order3}
Let K = Q$(\sqrt{D}$) (D=-1,-3) and let s,t $\in \mathcal{O}_K$ with s,t,d$ \ne 0$, $s^2 \ne t^2$.  Then there do not exist  a,b $\in \mathcal{O}_K$ such that
$$-s^2=a^3(a+2b)c^2$$
$$t^2 = b^3(b+2a)c^2$$
\end{prop}
\begin{proof}

Note that when $D=-1$, $-1$ is a square in K so this system has no solutions by \cref{order2=order3} (with d=1).  Hence we may assume $D=-3$.  It's now easy to check $\frac{a}{b} \not\in \{-2,-1,-\frac{1}{2}, 0, 1\}$ as otherwise a hypothesis is contradicted (or $-1$ is a square in K).  
Let s,t $\in K$ with s,t$ \ne 0$, $s^2 \ne t^2$ .  Suppose there exists a,b,c $\in K$ such that
$$-s^2=a^3(a+2b)c^2$$
$$t^2 = b^3(b+2a)c^2$$

As in the proof of \cref{order2=order3} we obtain

$$  (\frac{2st}{ab^2c^2})^2 =-2w(-2w-4)(-2w-1) $$

where $w := \frac{a}{b}$.  Hence we a get a K-point on the elliptic curve E with equation

$$ y^2 = x(x-4)(x-1) = x^3-5x^2+4x $$

Magma tells us E(K) has rank 0 and E(K)$_{tor} \approx C_2 \oplus C_4$ so $E(K) = \{(1-\alpha, -3-\alpha), (1-\alpha, -3+\alpha),(1+\alpha, -3+\alpha),(1+\alpha, 3-\alpha), (0,0), (4,0), (1,0)\}$ where $\alpha^2=-3$.  Hence $-2w = 0,1,4, 1\pm \alpha $
so $w = 0,-\frac{1}{2}, -2, \lambda$ or $\lambda^2$ where $\lambda = -\frac{1}{2}+\frac{\sqrt{-3}}{2}$ is a primitive cube root of unity.  We have already noted that $w \not = 0,-\frac{1}{2}, -2$.   If $w=\lambda$ then $a=\lambda b$ so  

$$s^2 =  b^3(b+2a) = b^3(b+2b\lambda)= b^4(1+2\lambda)$$ 
Hence  $1+2\lambda$ must be a square in K, but $N(1+2\lambda) = (1+2\lambda)(1+2\lambda^2) = 3$, which is not a square in $\mathbb{Q}$.  Hence there is no solution to the equation above with $w=\lambda$ and as $\overline{\lambda} = \lambda^2$ the case $w=\lambda^2$ follows similarly.

\end{proof}

\begin{prop}
\label{order3=order3}

Let K = Q$(\sqrt{D}$) ($D=-1,-3$) and let $a,b,c, a_0,b_0,c_0,d \in \mathcal{O}_K$ with d a non-square in K and $\frac{a}{b},\frac{a_0}{b_0} \not\in \{-2,-1, -1/2, 0,1  \}$.  Then the following system of equations has no $\mathcal{O}_K$-solutions:

$$   da^3(a+2b)c^2 =a_0^3(a_0+2b_0)c_0^2$$ 
$$db^3(b+2a)c^2 = b_0^3(b_0+2a_0)c_0^2$$

\end{prop}

\begin{proof}

Suppose a solution to the following system of equations exists satisfying the hypotheses above: 
$$   da^3(a+2b)c^2 =a_0^3(a_0+2b_0)c_0^2$$ 
$$db^3(b+2a)c^2 = b_0^3(b_0+2a_0)c_0^2$$

Note that we must have $a,b,a+2b \ne 0$ so we may divide:

$$ \left( \frac{a}{b}\right)^3 \frac{a+2b}{b+2a} = \left(\frac{a_0}{b_0} \right)^3 \frac{a_0+2b_0}{b_0+2a_0}$$

and letting $x =\frac{a_0}{b_0}$ and $y = \frac{a}{b}$ we get

$$ y^3 \frac{y+2}{1+2y} = x^3 \frac{x+2}{1+2x}$$

$$x^4+2x^3+2yx^4+4yx^3 - y^4-2y^3-2xy^4-4xy^3=0$$

Magma factored the polynomial on the l.h.s. into absolutely irreducible polynomials:

$$(x-y)(2x^3y+2xy^3+x^3+y^3+5x^2y+5xy^2+2x^2+2y^2+2x^2y^2+2xy)$$

Now if $x=y$, it follows that there is an $e \in K$ such that $a=be, a_0 = b_0e$.  
Substituting this into the system of equations above yields that $d$ is a square (a contradiction).  Hence there exists a K-rational point $(x,y)$ (with $x,y \not\in \{-2, \frac{-1}{2}, -1, 0, 1\}$)  on the curve C defined by $p_C(x,y)=0$ where
$$p_C = 2x^3y+2xy^3+x^3+y^3+5x^2y+5xy^2+2x^2+2y^2+2x^2y^2+2xy $$
Let $\overline{C}$ denote the projective closure of C.  Magma tells us $\overline{C}$ has genus 1 and is birational to the elliptic curve $E_C$ with defining equation 
 $$y^2+2xy+2y = x^3-x^2-2x$$
 (In fact, $E_C$ is birational to the curve with equation $y^2 = x^3+1$). Magma also provided a birational map:
 $$f: \overline{C} \to E_C$$ 
 $$(x,y,z) \mapsto (p^f_1,p^f_2,p^f_3)$$

 $p^f_1 = 2x^2y^2 + 3x^2yz + 4xy^2z - y^3z + x^2z^2 + 6xyz^2 - y^2z^2 + 2xz^3$\\
 
$p^f_2 = 2 x^2y^2 + 4xy^3 - x^2yz + 10xy^2z + 3y^3z - x^2z^2 + 7y^2z^2 -2xz^3 + 2yz^3$\\

$p^f_3 = y^4 + 3y^3z + 3y^2z^2 + yz^3$\\

Because f is a rational map, the only potential K-rational points of $E_C$ are the non-regular points of $E_C$ and $f^{-1}(E_C(K))$. 

Suppose $p^f_3=0$.  If z=0 then y=0 and if z=1 then y = 0 or -1.  Hence any non-regular point of f is of the form [x,0,0],[x,0,1] or [x,-1,1].  Evaluating $p^f_2=0$ at the last two points yields the following possibilities for non-regular points of f:  [1,0,0], [0,0,1], [-2,0,1], [-1,-1,1].  All such points do not give 'non-trivial' points on the curve C:  [1,0,0] does not correspond to a point on C and the other points fail to produce points (x,y) on C with $x,y \not\in \{ -2,-1,-1/2,0,1\}$.

In order to compute $f^{-1}(E_C(K))$ we must first compute $E_C(K)$.  Magma tells us $E_C(K)$ has rank 0 over  both quadratic fields.  See \cref{E_C} for a description of the torsion points.

\begin{table}
\caption{ K-Rational Points on $E_C$}

\begin{center}
    \begin{tabular}{| l | l | p{9cm} |}
    \hline
    \label{E_C}
    K &   $E_C(K)_{tor}$ & Points of $E_C(K)_{tor}$\\ \hline
    $\mathbb{Q}(\sqrt{-1})$ &  $ C_6$ & $[0,1,0], [-1,0,1], [0,-2,1], [0,0,1], [2,-6,1], [2,0,1]$ \\ \hline

$\mathbb{Q}(\sqrt{-3})$  & $C_2 \oplus C_6$ & 
     $[0,1,0], [-1,0,1], [0,-2,1], [0,0,1], [2,-6,1],[2,0,1],$ $[-1-\alpha,-3+\alpha,1], [-1-\alpha, 3+\alpha,1], \newline [-1+\alpha, -3-\alpha,1], [-1+\alpha, 3-\alpha,1], \newline [\frac{1}{2} (1-\alpha), \frac{1}{2}(-3+\alpha),1],  [\frac{1}{2} (1+\alpha), \frac{1}{2}(-3-\alpha),1]$\\ \hline
    \end{tabular}
\end{center}
\end{table}

To compute $f^{-1}([x,y,z])$, we form the ideal $<p_C, p^f_1-wx, p^f_2-wy,p^f_3-wz>$ and compute it's
Gr{\"o}bner basis.  Often, one can find basis elements that allow the system to be solved by hand (see \cref{GroebnerTable}).
\begin{table}
\caption{ Gr{\"o}bner basis data for determination of $f^{-1}(E(K))$ }
\begin{center}
    \begin{tabular}{| l | l | l |  p{10cm} |}
    \hline
    \label{GroebnerTable}
    
    Point P of E(K) & $f^{-1}(P)$ & Gr{\"o}bner basis elements \\ \hline
    [0,1,0] & [1,-1,1] & $xw - zw, yw + zw$  \\ \hline
    
    [-1,0,1] & [-1,1,1]  & $yw^2 - zw^2,xw + zw,$ \\ \hline
    [0,-2,1] &    $\emptyset$      &  $w^2$ \\ \hline
    [0,0,1]  &     [0,1,0]         & $xw,yw$ \\ \hline
    [2,-6,1] &   $\emptyset$   &   $w^2$ \\ \hline
    [2,0,1]  &  [0,-2,1]  &    $xw,yw^2 + 2zw^2 $ \\ \hline

    $[-1-\alpha,-3+\alpha,1]$  & $ [\alpha-1,2,0]$  & $xw + \frac{1}{2}(-q + 1)yw,zw$  \\ \hline
    $[-1+\alpha, -3-\alpha,1]$ &  $ [-\alpha-1,2,0]$  &   \\ \hline
    $[-1-\alpha, 3+\alpha,1]$  & $\emptyset$ & $w^2$      \\ \hline
    $[-1+\alpha, 3-\alpha,1]$  &  $\emptyset$  &   \\ \hline
    $[\frac{1}{2} (1-\alpha), \frac{1}{2}(-3+\alpha),1]$  &  $\emptyset$ & $w^2$  \\ \hline
    $[\frac{1}{2} (1+\alpha), \frac{1}{2}(-3-\alpha),1]$& $\emptyset$  &   \\ \hline

    \end{tabular}
\end{center}
\end{table}

Now we see each point of $f^{-1}(E_C(K))$ either does not correspond to a point on C, or fails to satisfy $x,y \not\in \{ -2,-1,-1/2,0,1\}$. 

\end{proof}


\section{Classification of Twists}

\label{Classification of Twists}

\begin{thm}
\label{twist}
Let K $= \mathbb{Q}(\sqrt{D}$) ($D = -1,-3$), 
d $\in$ K a non-square, and E/K an elliptic curve with full 2-torsion.

\begin{enumerate}

\item If $E(K)_{tor} \approx$ $C_2 \oplus C_8$ then $E^d(K)_{tor} \approx C_2 \oplus C_2$.

\item If $E(K)_{tor} \approx$ $C_2 \oplus C_6$ then $E^d(K)_{tor} \approx C_2 \oplus C_2$.

\item If $E(K)_{tor} \approx$ $C_4 \oplus C_4$ then $K = \mathbb{Q}(\sqrt{-1})$ and $E^d(K)_{tor} \approx C_2 \oplus C_2$.

\item If $E(K)_{tor} \approx$ $C_2 \oplus C_4$ then WLOG we may write $M=s^2$ and $N=t^2$. We have $E^d(K)_{tor} \approx C_2 \oplus C_2$ unless $K = \mathbb{Q}(\sqrt{-3}), d=-1,$ and there exists  v 
$\in$ $\mathcal{O}_K$ such that $s^2-t^2 = \pm v^2$, in which case $E^d(K)_{tor} \approx C_2 \oplus C_4$.

\item If $E(K)_{tor} \approx$ $C_2 \oplus C_2$ then $E^d(K)_{tor} \approx C_2 \oplus C_2$ for almost all d.  If for some d-twist,  $E^d(K)_{tor} \approx C_2 \oplus C_8$, $C_2 \oplus C_6$ or $C_4 \oplus C_4$ then it is the unique nontrivial twist of E. If for some d-twist,  $E^d(K)_{tor} \approx C_2 \oplus C_4$ then there is at most one other nontrivial twist of E and it is also isomorphic to $C_2 \oplus C_4$.

\end{enumerate}

\end{thm}

\noindent See \cite{Kwon}*{Thm 2} for the case $K = \mathbb{Q}$.

\begin{proof}

  Before we begin we note that in general, if $E$ and $E'$ are isomorphic elliptic curves and $d \in K$, then $E^d$ and $E'^d$ may only be isomorphic over $K(\sqrt{d})$.  But if $E$ and $E'$ are isomorphic via twist by a square in K, or by translation of the x-coordinate by an element of K, then one does in fact get an isomorphism of $E^d$ and $E'^d$ over K (and hence $E^d(K) \approx E'^d(K)$). 

1)  Suppose $E(K)_{tor} \approx$ $C_2 \oplus C_8$ .  Then WLOG we may assume there exists x,y,z $\in \mathcal{O}_K$ such that $M=x^4$, $N=y^4$ and $x^2+y^2 = z^2$. If $E^d(K)$ has a point of order 4 then by \cref{Ono} either (i) $dx^4,dy^4$ are both squares, or (ii) $-dx^4, dy^4-dx^4$ are squares, or (iii) $-dy^4, dx^4-dy^4$ are squares. In case (i) it follows that d is a square, a contradiction.  In case (ii), if $-dx^4$ is a square, then $d = -u^2$ for some u $\in \mathcal{O}_K$.  
If $dy^4-dx^4 = w^2$ for some $w \in \mathcal{O}_K$, then dividing by d yields $y^4-x^4 = -(wu^{-1})^2$ and multiplying by -1 we obtain
$x^4-y^4 = v^2$ for some $v \in \mathcal{O}_K $.  Note we have $x^4,y^4, x^4-y^4 \ne 0$ (as otherwise E is not smooth) and hence $x,y,v \ne 0$.
But this contradicts \cref{Fermat}, so $E^d(K)$ does not have a point of order 4.

Now looking above (or applying \cref{Ono}) we can write $M = s^2, N=t^2$ for some s,t $\in \mathcal{O}_K$. 
Hence if $E^d(K)$ has a point of order 3, there exists $a,b,c \in \mathcal{O}_K$ 
, $c \ne 0$ and $\frac{a}{b} \not\in \{-2,-1,-\frac{1}{2},0,1\}$ such that $ds^2=a^3(a+2b)c^2$ and $dt^2=b^3(b+2a)c^2$. But this contradicts \cref{order2=order3} so $E^d(K)$ does not have point of order 3.  Hence by  \cref{Najman} we conclude $E^d(K)\approx C_2 \oplus C_2 $.\\

2)  Suppose $E(K)_{tor} \approx$ $C_2 \oplus C_6$ .  Then we may assume  there exists $a,b,c \in \mathcal{O}_K$ with 
 $c \ne 0$ and $\frac{a}{b} \not\in \{-2,-1,-\frac{1}{2},0,1\}$ such that $M=a^3(a+2b)c^2$ and $N=b^3(b+2a)c^2$. If $E^d(K)$ has a point of order 4, then by Theorem \ref{Ono} 
(i) M, N are both squares (in $\mathcal{O}_K$), (ii) -M, N-M are both squares, or (iii) -N,M-N are both squares.  
Hence WLOG we can assume case (i) holds, so $dM =s^2, dN=t^2$ for some $s,t \in \mathcal{O}_K$ ($s,t \ne 0, s^2 \ne t^2$).  Hence 
$$     s^2=dM=da^3(a+2b)c^2 \mbox{ and } t^2=dN=db^3(b+2a)c^2$$
and multiplying by d:
$$     ds^2=a^3(a+2b)(cd)^2 \mbox{ and } dt^2=b^3(b+2a)(cd)^2$$
But this contradicts  \cref{order2=order3}.  Hence $E^d(K)$ does not have a point of order 4.  
See \cite{Kwon}*{Prop 4} for a similar argument in the case $K= \mathbb{Q}$.

Let K = Q$(\sqrt{D}$) ($D=-1$ or $-3$) and let d $\in \mathcal{O}_K$ be a nonsquare.  Suppose that E$(K)_{tor} \supseteq C_2 \oplus C_6$.  We will argue that the d-twist $E^d(K)_{tor} \not\supseteq C_3$.  We may assume (WLOG) that E is given by an equation of the form $$y^2 = x(x+M)(x+N)$$ where M,N $\in \mathcal{O}_K$. By Theorem \ref{Ono}, there exists $a,b,c \in \mathcal{O}_K$ with 
$\frac{a}{b} \not\in \{-2, \frac{-1}{2}, -1, 0, 1\}$ such that 
$$ M = a^3(a+2b)c^2 \mbox{ and } N= b^3(b+2a)c^2$$
Hence if E$_T^d(K) \supseteq C_3$ then there exists $ a_0,b_0,c_0 \in \mathcal{O}_K$  and $\frac{a_0}{b_0} \not\in \{-2, \frac{-1}{2}, -1, 0, 1\}$ such that 
$$   da^3(a+2b)c^2 = dM = a_0^3(a_0+2b_0)c_0^2$$ 
$$db^3(b+2a)c^2 = dN = b_0^3(b_0+2a_0)c_0^2$$
But this system has no solution by \cref{order3=order3}.
Hence we reach a contradiction and conclude  $E^d(K)_{tor} \not\supseteq C_3$.
Therefore by Theorem \ref{Najman} we conclude $E^d(K)\approx C_2 \oplus C_2 $.\\

4) Suppose $E(K)_{tor} \approx$ $C_2 \oplus C_4$ .  Then we may assume  WLOG that  there exists $s,t \in \mathcal{O}_K$ such that $M=s^2$ and $N=t^2$.  If $E^d(K)$ has a point of order 4, then either (i) dM, dN are both squares, (ii) -dM, dN-dM are both squares, or (iii)       -dN,dM-dN are both squares.  Case (i) is not possible since d is a nonsquare.  In case (ii) there exists u,v $\in$ $\mathcal{O}_K$ such that $-ds^2 = u^2$ and $dt^2-ds^2 = v^2$  so $d=-w^2$ for some $w$ $\in$ $\mathcal{O}_K$ 
Furthermore $D \ne -1$ as otherwise d is a square, so $K=\mathbb{Q}(\sqrt{-3})$. Hence $t^2-s^2 = z^2$ for some $z$ $\in$ $\mathcal{O}_K$. Case (iii) similarly yields that $s^2-t^2 = z^2$ for some $z$ $\in$ $\mathcal{O}_K$ and $d=-w^2$ for some $w$ $\in$ $\mathcal{O}_K$.  On other hand, if those condition are fulfilled then there is a point of order 4.

Hence if $K=\mathbb{Q}(\sqrt{-3})$ and $E(K)_{tor} \approx C_2 \oplus C_4$ then the 2-part of $E^d(K)$ can only be $C_2 \oplus C_2$, $C_2 \oplus C_4$ or $C_2 \oplus C_8$ by \cref{Najman} (clearly $E^d(K) \supseteq C_2 \oplus C_2$). But $C_2 \oplus C_8$ cannot occur by Case 1. 
Applying \cref{order2=order3} as in Case 2 shows that $E^d(K)$ has no point of order 3 and by  \cref{Najman} it has no point of order 5.\\


3)  Suppose $E(K)_{tor} \approx$ $C_4 \oplus C_4$.  Then we must have $K=\mathbb{Q}(\sqrt{-1})$.  On other other hand, as in the proof of 
Case 4 above, $E^d(K)$ cannot have a point of order 4.  Also as in the proof above, $E^d(K)$ cannot have a point order 3 or 5 so $E^d(K) \approx C_2 \oplus C_2$.\\


5) See the argument in \cite{Kwon}*{Section 4}.

\end{proof}

\section{Full 4-torsion}


\begin{lem}
\label{full4torsion1}
Let E be an elliptic curve defined over $K = \mathbb{Q}(\sqrt{D}) (D=-1, -3)$, $d \in  \mathcal{O}_K$ a nonsquare and let $L = K(\sqrt{d})$. 
If $E(K)_{tor} \approx  C_2 \oplus C_8$, then

\begin{itemize}

\item If D = -3, $C_4 \oplus C_4 \not\subseteq E(L)$.

\item If D = -1, there is a unique quadratic extension L of K such that $C_4 \oplus C_4 \subseteq E(L)$. WLOG, writing $M=x^4$ and $N=y^4$ where $x^2+y^2=z^2$, the extension is $L = K(\sqrt{x^2-y^2})$.

\end{itemize}

\end{lem}
\noindent See \cite{Kwon}*{Lemma 2} for a similar argument in the case $K = \mathbb{Q}$.
\begin{proof}

By  \cref{Ono}, WLOG we can assume there exists $ x,y,z \in \mathcal{O}_K$ such that  $M=x^4$, $N=y^4$ and $x^2+y^2  = z^2.$ If $C_4 \oplus C_4 \subseteq E(L)$, then every point of order 2 lifts, so in particular there exists a point P in $E(L)$ such that $2P=(-x^4,0)$.  Hence by \cref{lifting lemma}, $-x^4$ and $y^4-x^4$ are squares in $L$, so there exist $a,b \in \mathcal{O}_L$ such that $-x^4=a^2$ and $y^4-x^4 = b^2$. Therefore $-1 \in L$ so either (i) $ D=-1$ (if $-1 \in K$) or (ii) $D=-3$ and $d=-1 (L = K(\sqrt{-1}))$.  Also,
$$ b^2 = y^4-x^4 = (y^2-x^2)(y^2+x^2)  = (y^2-x^2)z^2$$
so $c^2 = y^2-x^2 \in K$ for some $c \in L$.  Writing $c=e+f\sqrt{d}$ ($e,f \in K$)  since $c^2 \in K$, we conclude that $c^2 = e^2$ or $df^2$.  Hence we get a solution to either $g^2 = y^2-x^2$  or $ dg^2 = y^2-x^2$ ($g \in \mathcal{O}_K$). 

In case (ii), we obtain either way $y^4-x^4 = z^2$ (or $x^4-y^4 = z^2$) for some $z \in \mathcal{O}_K$.  Furthermore we must have $x,y, x^4-y^4 \ne 0$, so $z \ne 0$ because E is nonsingular.  But this contradicts  \cref{Fermat}.
In case (i), we obtain either $y^4-x^4 = dz^2$ or $y^4-x^4 = z^2$ for some $d,z \in \mathcal{O}_K$ where $K = \mathbb{Q}(i)$.  The latter cannot occur by \cref{Fermat}.  

On the other hand, if we choose any Pythagorean triple $(x,y,z)$ in $\mathcal{O}_K$ with $x,y \ne 0$ and $x^4 \ne y^4$ and set $M=x^4, N=y^4$ then 
writing $y^4-x^4$ as $dz^2$ with d a nonsquare in $K$, we find (by  \cref{lifting lemma}) that $L = K(\sqrt{d})$) is the unique quadratic extension of K (in a fixed algebraic closure  $\overline{K}$ of $K$) such that  $C_4 \oplus C_4 \subseteq E(L)$.  Note that the argument in case (i) above implies that L/K is a nontrivial extension. 
\end{proof}
\begin{lem}
\label{full4torsion2}

Let E be an elliptic curve defined over $K = \mathbb{Q}(\sqrt{D}) (D=-1, -3)$, $d \in \mathcal{O}_K$, d a nonsquare and let $L = K(\sqrt{d})$ .   
Suppose $E(K)_{tor} \approx  C_2 \oplus C_4$. Then WLOG we may write $M=u^2, N=v^2$ for some $u,v \in \mathcal{O}_K$.  We have  $E(L)_{tor} \supseteq  C_4 \oplus C_4$ iff:

\begin{itemize}

\item D = -3, d=-1 and there exists a $w \in \mathcal{O}_K$ such that $v^2-u^2=\pm w^2$

\item D = -1 and L=K$(\sqrt{v^2-u^2})$

\end{itemize}

\end{lem}

\begin{proof}
WLOG we may assume (0,0) lifts over K and  write $M=u^2, N=v^2$ for some $u,v \in \mathcal{O}_K$.  If $E(L)_{tor} \supseteq  C_4 \oplus C_4$ then some other order 2 point lifts.  But by \cref{lifting lemma}, if $(-u^2,0)$ lifts over L then $-u^2$ and $u^2-v^2$ are squares in L, or equivalently, $-1$ and $u^2-v^2$ are squares in L.  As -1, $u^2-v^2$ $\in$ K, if $-1$ and $u^2-v^2$ are squares in L then each of them is a square in K or d times a square in K (and both will not be squares in K, as otherwise $E(K) \supseteq C_4 \oplus C_4$).

Now if $K=\mathbb{Q}(i)$, then -1 is a square in K so $u^2-v^2$ is not, and hence $L \supseteq K(\sqrt{u^2-v^2})$.  But $K(\sqrt{u^2-v^2})$ is quadratic over K so we must have $L= K(\sqrt{u^2-v^2})$.  Conversely, if $M=u^2, N=v^2$ for some $u,v \in \mathcal{O}_K$ and we define $L= K(\sqrt{u^2-v^2})$ then  $-u^2 = (iu)^2$ and $u^2-v^2$ are squares, so the result follows by the \cref{lifting lemma}.

If $K=\mathbb{Q}(\sqrt{-3})$ then as $-1$ is not a square in K, as above we must have $L=K(\sqrt{-1})$.  But if $u^2-v^2$ is a square in L, as $u^2-v^2 \in K$, we conclude $u^2-v^2 = w^2$ or $-w^2$ for some $w \in K$. Conversely, , if $M=u^2, N=v^2$ for some $u,v \in \mathcal{O}_K$ and we define $L= K(i)$ then  $-u^2 = (iu)^2$ and $u^2-v^2=w^2$ or $(iw)^2$ are squares, so the result follows by the lemma \ref{lifting lemma}. 
\end{proof}

\section{Classification of Growth}

Now we prove  \cref{main-1} and \cref{main-3}.  For clarity of proof, we prove each case of the main theorems simultaneously.  In the results below, the elliptic curves have full 2-torsion over the base field $K$ so each has a model of the form:
$$ y^2=x(x+M)(x+N)$$  for some $M,N \in K$.  We denote the curve above by $E(M,N)$. 
Our main results are the following:

\begin{thm}\label{main-1}
Let K $= \mathbb{Q}(\sqrt{-3}$) and $E/K$ an elliptic curve with full 2-torsion.

\begin{enumerate}

\item If $E(K)_{tor} \approx$ $C_2 \oplus C_8$ then $E(L)_{tor} \approx C_2 \oplus C_8$ for all quadratic extensions L of K.

\item If $E(K)_{tor} \approx$ $C_2 \oplus C_6$ then there is at most one quadratic extension $L'$ such that  $E(L')_{tor} \approx C_4 \oplus C_6$ and $E(L)_{tor} \approx C_2 \oplus C_6$ for all other quadratic extensions L of K.

\item If $E(K)_{tor} \approx$ $C_2 \oplus C_4$ then WLOG we may write $M=u^2, N=v^2$. There is a quadratic extension L of K such $E(L) \supseteq C_4 \oplus C_4$ iff $u^2-v^2= \pm w^2$ for some $w \in K$ and in this case $L=K(i)$ and $E(L) \approx C_4 \oplus C_4$.  There are at most two quadratic extensions L in which E(L) has a point of order (at least) 8 and in this case $E(L) \approx C_2 \oplus C_8$.

\item If $E(K)_{tor} \approx$ $C_2 \oplus C_2$ and L is a quadratic extension of K then one may have $E(L) \approx C_2 \oplus C_2, C_2 \oplus C_4, C_2 \oplus C_6,C_2 \oplus C_8,C_4 \oplus C_6$ or $C_4 \oplus C_4$.  If $L= K(\sqrt{d})$ for some $d \in K$ then this can be determined by replacing $E$ by $E^d$. The group $E(K)_{tor}$ grows in at most 3 extensions.

\end{enumerate}

\end{thm}

\begin{thm}\label{main-3}
Let K $= \mathbb{Q}(\sqrt{-1}$) and $E/K$ an elliptic curve with full 2-torsion.

\begin{enumerate}

\item If $E(K)_{tor} \approx$ $C_2 \oplus C_8$ then there is a unique quadratic extension $L'$ of K in which $E(L')_{tor} \approx C_4 \oplus C_8$ and $E(L)_{tor} \approx C_2 \oplus C_8$ for all other quadratic extensions $L$ of $K$.

\item If $E(K)_{tor} \approx$ $C_2 \oplus C_6$ then there are at most two quadratic extension $L'$ such that  $E(L')_{tor} \approx C_4 \oplus C_6$ and $E(L)_{tor} \approx C_2 \oplus C_6$ for all other quadratic extensions $L$ of $K$.

\item If $E(K)_{tor} \approx$ $C_2 \oplus C_4$ then WLOG we may write $M=u^2, N=v^2$. The quadratic extension $L=K(\sqrt{u^2-v^2})$ is the unique quadratic extension of K such $E(L) \supseteq C_4 \oplus C_4$ and in this case  $E(L) \approx C_4 \oplus C_4$. 
There are at most two other quadratic extensions $L'$ in which a point of order 4 (over K) lifts  and in this case $E(L') \approx C_2 \oplus C_8$.

\item If $E(K)_{tor} \approx$ $C_2 \oplus C_2$ and L is a quadratic extension of K then one may have $E(L) \approx C_2 \oplus C_2, C_2 \oplus C_4, C_2 \oplus C_6,C_2 \oplus C_8,C_4 \oplus C_6$ ,$C_4 \oplus C_4$ or $C_4 \oplus C_8$.  If $L= K(\sqrt{d})$ then this can be determined by replacing $E$ by $E^d$. The group $E(K)_{tor}$ grows in at most 3 extensions.

\item If $E(K)_{tor} \approx$ $C_4 \oplus C_4$ then  there is at most one quadratic extension L of K in which $E(K)_{tor}$ grows and in this case $E(L) = C_4 \oplus C_{8}$.

\end{enumerate}

\end{thm}

Let $\lambda$ denote a primitive 3rd root of unity. In the tables below we let $E := E(M,N)$.

\begin{table}[ht]

 \caption{Examples of Growth of Torsion  over $K=\mathbb{Q}(\sqrt{-1})$ 
} 

\centering 
\begin{center}
    \begin{tabular}{| l | l | l | l |}

    \hline
    
    \label{numberofext-1}

    $E(K)_{tor}$ & (M,N) & g(E) & $(d,E(K(\sqrt{d} )_{tor})$\\ \hline

       $C_2 \oplus C_8$ 
      & (81,256) & 1 & $(7,  C_4 \oplus C_8)$

  \\ \hline
    
       $C_2 \oplus C_6$ 
      &  $(64,189)$ & 2 & $(5, C_2 \oplus C_{12})$,$(21, C_2 \oplus C_{12})$

  \\ \hline
    
    $C_2 \oplus C_4$ 
      & (-1,1)& 1 & $(2,C_4\oplus C_4)$

  \\ \cline{2-4}
      & (1,4) & 2 & $(2 , C_2 \oplus C_8)$,$(3 , C_4 \oplus C_4)$

  \\ \cline{2-4}
     & (1,16)  & 3 & $( 3 , C_2 \oplus C_8)$ $( 5  , C_2 \oplus C_8)$ $( 15 , C_4 \oplus C_4)$

  \\ \hline
  
   $C_4 \oplus C_4$ 
      & (16,25) & 1 & $(5, C_4 \oplus C_8)$

  \\ \hline
  
   $C_2 \oplus C_2$ 
      & (-2,2) &1 & $(2, C_4 \oplus C_4)$
     \\ \cline{2-4}
     
      & (-2,1) &2  & $(2, C_2 \oplus C_4)$,$(3, C_2 \oplus C_4)$
     \\ \cline{2-4}
      & $(5\cdot 64,5 \cdot 189)$ &2 & $(5, C_4 \oplus C_6)$,$(105, C_2 \oplus C_4)$
     \\ \cline{2-4}
      & (25, 160) &3 & $(5,C_2 \oplus C_6 )$,$(10, C_2 \oplus C_4 )$, $(15, C_2 \oplus C_4)$

  \\ \hline

    \end{tabular}
   
\end{center}
\end{table}

\begin{table}[ht]

 \caption{ Examples of Growth of Torsion over $K=\mathbb{Q}(\sqrt{-3})$
 }
\centering 
\begin{center}
    \begin{tabular}{| l | l | l | l |}
    \hline
    \label{numberofext-3}

    $E(K)_{tor}$ &    (M,N)& g(E) & $(d,E(K(\sqrt{d})_{tor})$\\ \hline

       $C_2 \oplus C_6$ 
      &  $(64,189)$ & 1 &  $(21, C_4 \oplus C_{6})$

  \\ \hline
    
    $C_2 \oplus C_4$ 

 &  (-4,-3)  & 1 & $(3 , C_4 \oplus C_4)$
  \\ \cline{2-4}
     &  (16,25)  & 2 &$(5 , C_2 \oplus C_8)$,$(-1 , C_4 \oplus C_4)$
     
      \\ \cline{2-4}
    
     &  (25,-24) & 3 & $( 6 , C_2 \oplus C_8)$ $( -6  , C_2 \oplus C_8)$ $( -1 , C_4 \oplus C_4)$

  \\ \hline

   $C_2 \oplus C_2$ 
     & $(21\cdot 64, 21\cdot 189)$  & 1 & $(21, C_4 \oplus C_6)$
     \\ \cline{2-4}
     
     & (-2,2)  & 1 & $(2, C_2 \oplus C_4)$
     \\ \cline{2-4}
     & (-1,1)  & 2 & $(-1, C_2 \oplus C_4)$,$(2, C_2 \oplus C_4)$
     \\ \cline{2-4}
     & $(-1,\lambda)$  & 2 & $(3, C_4 \oplus C_4)$,$(-\sqrt{-3}, C_2 \oplus C_6)$
     \\ \cline{2-4}
     & (-1, -9)  & 3 & $(-1,C_2 \oplus C_8 )$,$(2, C_2 \oplus C_4 )$, $(-2, C_2 \oplus C_4)$

  \\ \hline

    \end{tabular}
   
\end{center}
\end{table}

\begin{table}[ht]

 \caption{Examples of Growth of Torsion over $K=\mathbb{Q}(\sqrt{-1})$ } 
 \centering 
\begin{center}
    \begin{tabular}{| l | l | l | l |}
    
    \hline
    
    \label{gro-1}

    (M,N) &  E(K)$_{tor}$ & E(K($\sqrt{d}$))$_{tor}$& d \\ \hline
    
    (-2,2) & $C_2 \oplus C_2$ & $C_4 \oplus C_4$  & 2\\ \hline
    
   ($5 \cdot 64, 5 \cdot 189$)  & $C_2 \oplus C_2$ & $C_4 \oplus C_6$ & 5\\ \hline
    
    ($7 \cdot 3^4, 7 \cdot 4^4$) &  $C_2 \oplus C_2$   &  $C_4 \oplus C_8$ & 7 \\ \hline
    
    (1,4)  &    $C_2 \oplus C_4$       & $C_4 \oplus C_4$ & -3  \\ \hline
    (64,189) &   $C_2 \oplus C_6$   &   $C_4 \oplus C_6$ & 5 \\ \hline
    ($3^4, 4^4$)  &  $C_2 \oplus C_8$  &   $C_4 \oplus C_8$ & 7 \\ \hline
    ($4^2,5^2$)  & $C_4 \oplus C_4$  &   $C_4 \oplus C_8$ & 5 \\ \hline

    \end{tabular}
\end{center}
\end{table}

\begin{table}[ht]
 \caption{Examples of Growth of Torsion over $K=\mathbb{Q}(\sqrt{-3})$ } 
 \centering 
\begin{center}

    \begin{tabular}{| l | l | l | l |}

    \hline
    
    \label{gro-3}

    (M,N) &  E(K)$_{tor}$ & E(K($\sqrt{d}$))$_{tor}$& d \\ \hline
    
    (-1, $\lambda$) &  $C_2 \oplus C_2$ & $C_4 \oplus C_4$ & -1\\ \hline
    
    ($21 \cdot 64$, $21 \cdot  189$) &  $C_2 \oplus C_2$ & $C_4 \oplus C_6$ & 21 \\ \hline

    ($4^2,5^2$)  &     $C_2 \oplus C_4$       & $C_4 \oplus C_4$ & -1  \\ \hline
    (64,189) &    $C_2 \oplus C_6$   &  $C_4 \oplus C_6$ & 21 \\ \hline

    \end{tabular}
\end{center}
\end{table}
\newpage
\noindent Note only one example in \cref{gro-1} and \cref{gro-3} is not defined over $\mathbb{Q}$.  In that case there is in fact no example defined over $\mathbb{Q}$.


\subsection{$E(K)_{tor} 
\approx C_2 \oplus C_8$}

\begin{proof}

1) Suppose $E(K)_{tor} \approx C_2 \oplus C_8$
WLOG we can assume there exists $u,v,w \in \mathcal{O}_K$ such that  $M=u^4$ and $N=v^4$ and $u^2+v^2  = w^2$.  Let L $=$ K$(\sqrt{d})$ where d $\in$ K is a nonsquare.  By \cref{twist}, it follows that $E(L) \approx C_2 \oplus C_8, C_4 \oplus C_8, C_2 \oplus C_{16}$ or $C_4 \oplus C_{16} $.  
We have the following points on E:
$$P'=(-u^4,0)$$
$$P = (uv(u+w)(v+w), uvw(u+v)(v+w)(u+w))$$
$$P'+P = (uv(u+w)(v-w), uvw(u-v)(v-w)(u+w))$$
where $P$ is obtained by applying \cref{lifting lemma}.  The same argument as in \cite{Kwon}*{p.156} shows that $P$ or $P'+P$ lifts over L.  Suppose P lifts over L (the case in which $P'+P$ lifts  is similar).
Explicit calculation shows:
$$P_x = uv(u+w)(v+w)$$
$$P_x+u^4 =  uw(u+v)(w+v)$$
$$P_x+v^4 = vw(v+u)(w+u)   $$

Arguing as in \cite{Kwon}*{p.156} we see one of the above expressions must be a square in K.

If $P_x$ is a square then we have $u,v,w,z \in K$     such that 

$$z^2 = uv(u+w)(v+w)$$
$$ u^2+v^2=w^2$$

\noindent Note that if (u,v,w)=(-1,0,1) then $v=0$ and hence our curve is not smooth, contradicting our hypothesis. By \cref{circle} we can write any other solution to the above system (twisting by a square if necessary) $(u,v,w) =  (1-m^2, 2m, 1+m^2)$.  Again since our curve is smooth by hypothesis, we see that $m \ne 0,\pm1,\pm i$. Evaluating the above equation yields the following curve C:

$$z^2 = (1-m^2)(2m)(2)(m+1)^2 = m(1-m^2)(2m+2)^2$$

\noindent Letting $E_0$ denote the elliptic curve with equation $y^2 = x^3-x$, we have a birational map $f: C \rightarrow E_0$ given by $(m,z) \mapsto (-m,\frac{z}{2m+2})$.  
Magma tells us $E_0(K)$ has rank 0 and $E_0(K)_{tor} \approx C_2 \oplus C_4$.  In fact $E(K)_{tor} = \{ \infty, (0,0), (\pm 1, 0), (i, \pm(1-i) ), (-i, \pm(1+i)  )  \}$.
Hence all rational points (m,z) on C satisfy $m=0,\pm 1$ or $\pm i$, a contradiction.  Similarly when 
$ K=\mathbb{Q}(\sqrt{-3})$ we have $ E_0(K) \approx  C_2 \oplus C_2$ (the rank is 0), so $m=0, \pm 1$, a contradiction. 

If $P_x+u^4$ is a square, then we have $u,v,w,z \in \mathcal{O}_K$     such that 

$$z^2 = uw(u+v)(w+v)$$
$$ u^2+v^2=w^2$$

\noindent Again by \cref{circle} (noting that (-1,0,1) does not yield a solution satisfying the hypotheses), WLOG we can write any solution to the above system $(u,v,w) =  (1-m^2, 2m, 1+m^2)$.  By our hypotheses concerning u,v,w we see that $m \ne 0,\pm1$.
Evaluating the above equation yields the following curve C:

$$z^2 = (1-m^2)(1+m^2)(1+2m-m^2)(m+1)^2 $$

Let $E_1$ denote the curve $$y^2 = (1-x)(1+x)(1+x^2)(1+2x-x^2)$$

\noindent and let J denote its Jacobian.  We wish to classify all K-rational points on this curve.  Magma tells us $E_1$ is smooth (hence hyperelliptic) of genus 2.  Suppose K$= \mathbb{Q}(\sqrt{-1})$ (the case K$= \mathbb{Q}(\sqrt{-3})$ is simpler).  Magma tells us J(K) has rank 0 and 
$$J(K)_{2-part} \approx C_2 \oplus C_{2} \oplus C_{2}$$
$$J(\mathbb{Q})_{tor} \approx C_2 \oplus C_{2} \oplus C_{5} $$
\noindent and furthermore Magma gives us a list of generators.  Together this allows us to generate 40 points on J(K).  

On the other hand, we will argue that J(K) has size at most 40.  
Note that 3 remains prime and 5 splits in $\mathcal{O}_K$. Let $\mathfrak{B}$ denote a prime above 5 in $\mathcal{O}_K$. 
The curve $E_1$  remains non-singular mod 3 and $\mathfrak{B}$ and hence J has good reduction at those primes. Magma tells us $|J(F_9)| = 80 $ and $|J(F_5)| = 40 $.
Since the reduction map  $J(K)_{tor} \rightarrow J(F_5)$ (modulo $\mathfrak{B}$) is injective on the prime-to-5 part, we see the prime-to-5-part has size at most 40 (and hence is a 2-group of size at most 8).  On the other hand, under reduction mod 3 the 5-part of $J(K)_{tor}$ injects into $J(F_9)$ and so can be of size at most 5 . It follows that $|J(K)|=40$.

Magma embeds $E_1$ in (1,3,1)-weighted projective space, yielding two points at infinity $\infty_1 = (1,-1,0)$ and $\infty_2 =  (1,1,0)$.  Now $E_1$ embeds into J via the Abel-Jacobi map:

$$ E_1 \rightarrow J$$
$$ P \mapsto P - \infty_1$$

\noindent Note that the image of P is $P - \infty_1 = P + \infty_2 - \infty_1 - \infty_2$ and the latter expression is the unique coset representative that Magma selects. Among the 40 coset representatives that Magma found, 7 are of the above form and the corresponding points P are $(0,\pm1,1), (1,1,0),(0,1,1),(\pm1,0,0),(\pm i,0,1)$.  Again, any such point corresponds to u,v,w which contradict the hypothesis that our curve is smooth.

If $P_x+v^4$ is a square then one obtains the same system as in the previous case but with u and v interchanged.  Hence using the parametrization $(u,v,w) = (2m, 1-m^2, 1+m^2)$ 
we obtain the same hyperelliptic curve above.
Note that $u^2+(-v)^2=w^2$ and $P'+P$ is related to P by the substitution -v for v. Hence if $2Q=P'+P$ then one obtains the same (hyper)elliptic curves as in the case $2Q=P$.  We conclude that we cannot have $E(L) \approx C_2 \oplus C_{16}$.





Now if $E(L) \approx C_4 \oplus C_{16}$, then in particular $C_4 \oplus C_{4} \subset E(L)$, so by  \cref{full4torsion1}, $K=\mathbb{Q}(i)$ and $d = \sqrt{v^2-u^2}$.
In thise case $E(L)$ has 16 points of order 8.  Applying  \cref{lifting lemma} to obtain all 16 points of order 8, we find that (replacing $w$ by $-w$ or $u$ by $-iu$ if necessary) we may asssume any point P of order 8 has $P_x = uv(u+w)(v+w) \mbox{ or } uv(u+w)(v-w)$.
But in the case $E(L) \approx C_2 \oplus C_{16}$ we already showed that points P of the above form do not lift.

\end{proof}
\subsection{$E(K)_{tor} \approx C_2 \oplus C_6$ }
\begin{proof}

2) Now suppose $E(K)_{tor} \approx C_2 \oplus C_{2}\oplus C_{3}$.  Again by \cref{twist} and \cref{kwon} we have the only possibilities are $E(L)_{tor} \approx C_2 \oplus C_{2}\oplus C_{3}$ or $C_2 \oplus C_{4}\oplus C_{3}$ or $C_4 \oplus C_{4}\oplus C_{3}$.  If $E(L)$ has a point of order 4, 
then either (i) M, N are squares in L (if (0,0) lifts), (ii) -M,N-M are squares in L (if (-M,0) lifts) or (iii) -N, M-N are squares in L (if (-N,0) lifts).  It follows that each order 2 point lifts in at most one quadratic extension.  Furthermore note that if at least two order 2 points lift, then all three lift. 

If (0,0) lifts in $L = K(\sqrt{d})$ then M and N are squares in L and hence WLOG $M=s^2$ or $ds^2$ and $N=t^2$ or $dt^2$ for some $s,t \in \mathcal{O}_K$. Now if $M=s^2$ and $N=t^2$ then E(K) has a point of order 4, contradicting our hypothesis.  Also, if $M=ds^2$ and $N=dt^2$ then as E(K) has a point order 3, we get a contradiction by \cref{order2=order3}.  Hence WLOG we can write $M=ds^2$ and $N=t^2$ and $L = K(\sqrt{d})$ is the unique quadratic extension in which (0,0) lifts.  Noting that $E(-M, N-M) \approx E(M,N)$ we can apply the same argument and conclude that if -M, N-M are square in L (that is, (-M,0) lifts) that one is a square (and the other d times a square).  Similar reasoning applies to -N, M-N 
(see \cite{Kwon}*{p.157} for a similar argument).\\

Suppose $K = \mathbb{Q}(\sqrt{-3}$). We will argue that:

\noindent (I) If (0,0) lifts in a quadratic extension $L = K(\sqrt{d})$, then none of -M,N-M, -N, M-N are squares in K.  

If this holds, then by the argument above, it must be that (-M,0) and (-N,0) dont lift in any quadratic extension of K, Hence L is the unique quadratic extension in which the torsion grows and $E(L) \approx C_2 \oplus C_{4} \oplus C_3$.  In particular, E(L) cannot achieve full 4-torsion.
We would also like to show that:

\noindent (II) If (-M,0) lifts in some quadratic extension L of K then (0,0) and (-N,0) dont lift in any quadratic extension of K.

\noindent (III) If (-N,0) lifts in some quadratic extension L of K then (0,0) and (-M,0) dont lift in any quadratic extension of K.

But these latter two claims ((II) and (III)) follow from the first (I) by the argument in \cite{Kwon}*{p.158}. 
Now let's prove the first claim.  First if $-M = v^2$ for some $v \in K$ then $ds^2=-v^2$ so WLOG $d=-1$.
As E has a point of order 3, we obtain a nontrivial solution to:
$$-s^2=M = a^3(a+2b)c^2$$
$$t^2 =N = b^3(b+2a)c^2$$
contradicting  \cref{modifiedorder2=order3} .  

If $M-N = v^2$ for some $v \in K$. Then as $E(-N, M-N) \approx E(M,N)$ has a point of order 3, we obtain a nontrivial solution to:
$$-t^2= -N = a^3(a+2b)c^2$$
$$v^2 =M-N = b^3(b+2a)c^2$$
contradicting \cref{modifiedorder2=order3}.

If $N-M = v^2$ for some $v \in K$. Then as $E(-N, M-N) \approx E(M,N)$ has a point of order 3, we obtain a nontrivial solution to:
$$-t^2= -N = a^3(a+2b)c^2$$
$$-v^2 =M-N = b^3(b+2a)c^2$$
contradicting \cref{order2=order3}. Finally if -N is a square in K then it follows -1 is a square in K, a contradiction.

Now suppose $K = \mathbb{Q}(i)$ and some point of order 2 lifts.  WLOG we may assume this point is (0,0) and the same proof as above shows (I) holds with the exception that -N is of course a square.  As (I) implies (-M,0) never lifts, we still cannot have $E(L) \supset C_4 \oplus C_4$ in any quadratic extension L, but (-N,0) lifts in the extension $K(\sqrt{M-N})$.  Hence we have at most two quadratic extensions L of K such that $E(L) \approx  C_4 \oplus C_6$.  For example, the curve E(64,189) satisfies $E(K)_{tor} \approx C_2 \oplus C_6$ but $E(K(\sqrt{d}))_{tor} \approx C_2 \oplus C_{12}$ for $d=5,21$.

\end{proof}
\subsection{$ E(K)_{tor} \approx C_2 \oplus C_4$}

\begin{proof}

3)Now suppose $E(K) \approx C_2 \oplus C_{4}$.  
First suppose $K = \mathbb{Q}(\sqrt{-3}), L = K(i), M=u^2,N=v^2$ and $u^2-v^2 = w^2$ for some $u,v,w \in K$.
  By  \cref{full4torsion2}, we have $E(L) \supseteq C_4 \oplus C_{4}$ 
and we will argue that in this case there is no point of order 8.
Consider the following three points of order 4:
$$Q = (uv, uv(u+v))$$
$$Q' = (-u(u+w),iuw(u+w))$$
$$Q+Q' = (-v(v+iw), -vw(v+iw)) $$
Arguing as in \cite{Kwon}*{p.159} we find that $Q,Q'$ or $Q+Q'$ lifts.  
We will consider each case:
If Q lifts, then by \cref{lifting lemma} $uv, u(u+v), v(u+v)$ are squares in L.  As they lie in K, it follows that each is $\pm$ a square in K. If $uv= -c^2 $ note (-1,0,1) is not a solution to our problem as our curve is smooth. By \cref{circle}, WLOG we may use the parametrization $(u,v,w) = (1+m^2, 2m, 1-m^2)$ of $u^2 = v^2+w^2$ (note $m \ne 0, \pm 1$ as our curve is nonsingular) yielding $$-c^2 = (1+m^2)(2m)$$ And multiplying by 4 yields $$ (2c)^2 = (4+(-2m)^2)(-2m)$$

So we get a K-point on the elliptic curve $E_0$ with model:
$$ y^2 = x^3+4x $$ Magma tells us $E_0(K)_{tor} \approx C_4 $ and $E_0$ has rank 0, so the affine points of $E_0(K)$ are $(0,0), (2,\pm 4)$.  Hence  $-2m = 0, 2$ so $ m=0$ or $-1  $, a contradiction.  If $uv = c^2$ then one obtains the same curve $E_0$ as above, hence no nontrivial solutions.  If $Q'$ lifts then by \cref{lifting lemma} $-u(u+w), -uw, -2(u+w)$ are squares in $L$ and as above we conclude $-uw$ is $\pm$ a square in $K$.  If $-uw = -c^2$ then the parametrization above yields:

$$ c^2 = (1+m^2)(1-m^2) = 1-m^4$$  and multiplying by $m^2$ yields:

$$ (mc)^2 = m^2-m^6 = (m^2)-(m^2)^3 = (-m^2)^3 - (-m^2)$$ Letting $E_0$ denote the elliptic curve:

$$y^2 = x^3-x$$ Magma tells us $E_0(K)$ has  rank 0 and $E(K)_{tor} \approx  C_2 \oplus C_{2}$ so the only affine K-points are $(1,0),(0,0)$ and $(-1,0)$ .  Hence we get $-m^2=0, \pm 1$ so $m = 0, \pm 1$, a contradiction.
Finally, suppose $Q+Q'$ lifts.  By Thm \ref{lifting lemma}, it follows that 
-iwv is a square in L. 
$$-iwv = (a+bi)^2$$
$$ -iwv = a^2-b^2 + 2abi$$
From this it follows that $a=\pm b$ and so $wv = \pm 2a^2$.  Using the parametrization above yields $$\pm a^2 = m(1-m^2)$$

Let E denote the elliptic curve $$y^2 = x^3-x$$ We wish to classify the points on $E(K)$ and E$^{-1}(K)$ but note that $E^{-1} \approx$ E.  Magma tells us E(K) has rank 0 and $E(K)_{tor} \approx  C_2 \oplus C_{2}$ so the only affine K-points are (1,0),(0,0) and (-1,0) . In any case, we get $a=0$ which implies $w=0$ or $v=0$, yielding a contradiction.\\

Now suppose $K = \mathbb{Q}(\sqrt{-1}),M=u^2,N=v^2$, L = K(w) where $u^2-v^2 = w^2$ for some $u,v \in \mathcal{O}_K$ with 
[L:K]=2. 
 By \cref{full4torsion2}, we have $E(L) \supseteq C_4 \oplus C_{4}$ and so by  \cref{kwon} and  \cref{twist} 
$E(L) \approx C_4 \oplus C_{4}$ or $  C_4 \oplus C_{8}  $.

$$Q = (uv, uv(u+v))$$
$$Q' = (-u(u+w),iuw(u+w))$$
$$Q+Q' = (-v(v+iw), -vw(v+iw)) $$
Note $2Q = (0,0)$ and $2Q' = (-u^2,0)$ so $Q$, $Q'$ are independent order 4 points and hence form a basis for $ C_4 \oplus C_{4} \subseteq E(L)$. Now if some point of order 4 lifts in L, then as in the previous case we can say WLOG that $Q,Q'$ or $Q+Q'$ lifts.  Two cases can be dismissed quickly:

If $Q'$ lifts then by  \cref{lifting lemma} $-u(u+w) + u^2 = -uw$ is a square in L and if $Q+Q'$ lifts $-v(v+iw)+v^2 = -ivw$ is a square in L. In both cases, we find a square in L of the form $cw$ with $c \in K$ (note $i \in K$).  Now we make the following observation:  if $F \subset F'$ is a quadratic extension of fields and $F$ contains a square root of $-1$, then the (nonzero) squares in $F$ cannot be of the above form:  if $F' = F(w)$ then a typical element in $F'$ can be written $a+bw$ with $a,b \in F$.  So we find $(a+bw)^2 = a^2+b^2w^2 + 2abw$ will be of the form mentioned above only if $a^2+b^2w^2 = 0$. But in this case it follows that either $b=0$ (and hence $a=0$) or $w= \pm \frac{ia}{b} \in K$, a contradiction.  Therefore neither $Q'$ nor $Q+Q'$ lift over L.

If $Q$ lifts, then by  \cref{lifting lemma} $uv, u(u+v), v(u+v)$ are squares in L.  As they lie in K, it follows that each is a square in K or $u^2-v^2$ times a square in K. As they cannot all be squares in K, it follows that exactly one is a square in K and hence the other two have the same non-square part (as their product is  square) \cite{Kwon}*{p.160}.  Hence there is at most one extension in which $E(L) \approx C_4 \oplus C_{8}$.
We will prove that Q does not lift in L, so there are in fact no extensions in which $E(L) \approx C_4 \oplus C_{8}$. 
As described above, if $Q$ lifts then  one of the three sets are squares in K:

$$ (i) \space uv, u(u+v)(u^2-v^2), v(u+v)(u^2-v^2)    $$
$$ (ii) \space uv(u^2-v^2), u(u+v), v(u+v)(u^2-v^2)    $$
$$ (iii) \space uv(u^2-v^2) , u(u+v)(u^2-v^2), v(u+v)   $$  Modulo squares, this is equivalent to:
$$ (i) \space uv, u(u-v)   $$
$$ (ii) \space u(u+v), v(u-v)    $$
$$ (iii) \space  u(u-v), v(u+v)   $$ 

If (i) holds then by  \cref{lifting lemma} the point P of order 4 in E(K) with $P_x = -uv$ lifts in E(K) (a contradiction).  If (ii) holds then letting $z = \frac{u}{v}$ we obtain $$ a^2=z(z+1)  $$ $$b^2= z-1$$ and taking products yields $$(ab)^2 = z^3-z$$  Likewise if (iii) holds $$ a^2=z(z-1)  $$ $$b^2= z+1$$ and taking products yields $$(ab)^2 = z^3-z$$  Hence in either case we obtain a point on the elliptic curve $$y^2=x^3-x$$  But as noted in the case $E(K)_{tor} \approx C_2 \oplus C_8$, the only K-rational points $(x,y)$ on this curve satisfy $x=0, \pm 1, \pm i$.
If $\frac{u}{v} = 0, \pm 1$ then E is not smooth, a contradiction.  If $\frac{u}{v} =\pm i$ then $u^2=-v^2$.  But for any nonzero $v \in K$, the curve $E(v^2,-v^2)$ is isomorphic over K to $E(1,-1)$ and the latter curve does not have a point of order 8 in $K(\sqrt{2})$.

Now suppose $E(K) \approx C_2 \oplus C_{4}$ ($K = \mathbb{Q}(\sqrt{-1})$ or $\mathbb{Q}(\sqrt{-3})$) and the conditions in \cref{full4torsion2} do not hold (i.e. E(L) $\not \supseteq  C_4  \oplus C_{4}$).
Again by \cref{twist} (and  \cref{full4torsion2}) we have the only possibilities are $E(L) \approx C_2 \oplus C_{4}$, $C_4 \oplus C_{4}$, $C_2 \oplus C_{8}$ or $C_4 \oplus C_{8}$ and as noted, two of these cases (full 4-torsion) cannot occur.  So now we just need determine in which quadratic extensions of K a point of order 4 lifts.

WLOG assume $M=u^2$, $N=v^2$. 
 There are 4 order 4 points in E(K) (each a lift of (0,0)) and we can describe them explicitly:
$( uv, \pm uv(u+v)),( -uv, \pm uv(u-v))$.  Hence if P is an order 4 point in E(K) then $P_x = uv $ or $-uv$ and furthermore these cases cannot occur simultaneously (as if $E(L) \approx C_2 \oplus C_{8}$ then Im([2]) contains only 2 order 4 points, and it is closed under inverses).

If $P_x = uv$ and P lifts in L$=K(\sqrt{d})$ with d a nonsquare in $K$ then $uv, u(u+v), v(u+v)$ are squares in L, so each is a square in K or d times a square in K. Since their product is a square in K, and not all of them can be squares in K (since otherwise we have a point of order 8 in E(K)), it follows that exactly one of $uv, u(u+v), v(u+v)$ is a square in K.
If $uv$ is a square then  all three expressions are squares in $K(\sqrt{u(u+v)}) = K(\sqrt{v(u+v)})$. Likewise if $u(u+v)$ is a square in K then all three expressions are squares in $K(\sqrt{uv})$ (and similarly if $v(u+v)$ is a square in K).
If $P_x = -uv$ then one obtains the same results above but with v replaced by -v. Hence the growth $C_2 \oplus C_4$ to $C_2 \oplus C_8$ occurs in at most two quadratic extensions  (see \cite{Kwon}*{p.160} for a similar discussion).  If $K=\mathbb{Q}(i)$, then uv is a square iff -uv is a square. Hence if we choose uv to be a square and if $E(K)_{tor} \approx C_2 \oplus C_4$ then each order 4 points will lift in some extension.  When $u=1, v=4$ then $E(1,16)_{tor} \approx C_2 \oplus C_4 $ and $E(K(\sqrt{d}))_{tor} \approx C_2 \oplus C_8$ for $d=-3,5$.  We also have $E(K(\sqrt{-15}))_{tor} \approx C_4 \oplus C_4$ so $E(1,16)$ grows in 3 quadratic extensions.  If $K=\mathbb{Q}(\sqrt{-3})$, suppose $E(K)_{tor} \approx C_2 \oplus C_4$ and E obtains full 4-torsion in a quadratic extension.  By \cref{full4torsion2} there is a $w \in K$ such that $u^2-v^2 =w^2$.  Hence $(u+v)(u-v) = w^2$ so $u(u+v)$ is a square iff $u(u-v)$ is a square (and in this case each order 4 point lifts in some extension).  Consider the system
$$  a^2 = u(u+v) $$  $$ u^2-v^2 = w^2$$
Using the parametrization $(1+m^2, 2m, 1-m^2)$ for (u,v,w) (note we must have u,v,w nonzero above) yields
$$b^2 = 1+m^2$$
for some $b \in K$.  Using the parametrization $(b,m) = (\frac{1+n^2}{2n}, \frac{1-n^2}{2n})$ and letting say n=2 yields (after twisting) E(25,-24).  We have $E(K)_{tor} \approx C_2 \oplus C_4 $ and $E(K(\sqrt{d}))_{tor} \approx C_2 \oplus C_8$ for $d=\pm 6$.  We also have $E(K(i))_{tor} \approx C_4 \oplus C_4$ so $E(K)_{tor}$ grows in 3 quadratic extensions.

\end{proof}

\subsection{$E(K)_{tor} \approx C_4 \oplus C_4$}

\begin{proof}

5) If $E(K) \approx C_4 \oplus C_{4}$  where K is a quadratic extension then  we must have $K = \mathbb{Q}(i)$.  It follows from \cref{twist}
that  $E(L) \approx C_4 \oplus C_{4}$, $C_4 \oplus C_{8}$ or $C_8 \oplus C_{8}$.  WLOG we may write $ M=u^2,N=v^2$, $u^2-v^2 = w^2$ for some $u,v,w \in \mathcal{O}_K$.
$$Q = (uv, uv(u+v))$$
$$Q' = (-u(u+w),iuw(u+w))$$
$$Q+Q' = (-v(v+iw), -vw(v+iw)) $$
As above $Q$, $Q'$ are independent order 4 points and hence form a basis for $ C_4 \oplus C_{4} \subseteq E(L)$. Also, if some point of order 4 lifts in L then we can say WLOG that $Q,Q'$ or $Q+Q'$ lifts.

WLOG we may now assume the parametrization of $w^2+v^2=u^2$ of the form $(1-m^2, 2m, 1+m^2)$. Note in this case that $u+v = (1+m)^2$ and $w-iv = (im-1)^2$ are squares in K  and $u+w =2$.  First we claim that $uv$, $uw$ and $ivw$ are never squares in K.  If $uv =c^2$ for some $c \in K$ then:
$$c^2 = uv = (1+m^2)2m$$

And multiplying by 4 yields:
$$(2c)^2 = (4+(2m)^2)(2m)$$

So we get a K-point on the elliptic curve $E_0$:
$$y^2 = x^3+4x$$  Magma tells us $E_0$ has rank 0 and $E(K)_{tor} \approx  C_2 \oplus C_{4}$ so the only affine K-points are $(0,0),(\pm 2i,0), (2,\pm 4), (-2, \pm 4i)$.  Hence we get $2m=0, \pm 2, \pm2i$ so $m = 0, \pm 1, \pm i$, a contradiction.

If $uw=c^2$ for some $c \in K$, then $$c^2 = uw = (1+m^2)(1-m^2)= 1-m^4$$
And multiplying by $m^2$ yields:

$$(mc)^2 = (m^2) - (m^2)^3 = (-m^2)^3 -(-m^2)$$ So we get a K-point on the elliptic curve $E_0$:

$$y^2 = x^3-x$$  Magma tells us this curve has has rank 0 and $E(K)_{tor} \approx C_2 \oplus C_{4}$  and list of 8 points shows that the x-coordinates of points in $E(K)$ are $0,\pm1, \pm i$. Hence m satisfies $-m^2 = 0,\pm1, \pm i$ so $m=0, \pm1,\pm i$, a contradiction (as our curve E is nonsingular by hypothesis).  

If $ivw = c^2$ for some $c \in K$ then $$ c^2 = 2im(1-m^2)$$ and multiplying by -4 yields

$$ (2ic)^2 = (2im)(-4+(2im)^2)$$

So we get a K-point on the elliptic curve $E_0$:

$$y^2 = x^3-4x$$

Magma tells us this curve has has rank 0 and $E(K)_{tor} \approx C_2 \oplus C_{2}$  and list of 4 points shows that the x-coordinates of points in $E(K)$ are $0,\pm 2$.  Hence m satisfies $2im = 0, \pm 2$ so $m = 0, \pm i$, a contradiction (as E nonsingular by hypothesis).\\

Now if $Q$ lifts (in some extension) it follows that exactly one of $uv, u(u+v), v(u+v)$ is a square in K and hence one of $uv, u, v$ is a square in K (as $u+v$ is a square). If $Q'$ lifts in some extension then exactly one of $uw, 2u, 2w$ is a square in K. Finally if $Q''$ lifts in some extension then exactly one of $ivw, w, iv$ is a square in K.  We have already argued that $uv, uw$ and $ivw$ are never squares in K.

Suppose that $Q$ lifts in some extension and $Q'$ lifts in some (possibly different) extension.  The only way this can occur is if one of the following pairs consists of squares in K: (i) $u$ and $2u$ (ii) $u$ and $2w$ (iii) $v$ and $2u$ or (iv) $v$ and $2w$.  In case (i) we find 2 is a square in K, a contradiction.  
In case (ii), we get K-solutions to the system 
$$a^2 = 1+m^2$$
$$ b^2 = 2(1-m^2)$$
Taking products yields:
$$(ab)^2 = 2(1-m^4)$$
Multiplying both sides by $4m^2$:
 $$(2mab)^2 = 8(m^2-m^6)$$
 $$(2mab)^2 = (-2m^2)^3-4(-2m^2)$$
 So we get a point on the curve $$y^2 = x^3-4x$$
 
Magma tells us this curve has has rank 0 and $E(K)_{tor} \approx C_2 \oplus C_{2}$  and list of 4 points shows that the x-coordinates of points in $E(K)$ are $0,\pm 2$.  Hence m satisfies $-2m^2 = 0, \pm 2$ so $-m^2 = 0, \pm 1$ and $m = 0, \pm 1, \pm i$. But then $m=0,\pm 1, \pm i$ and each such choice for m yields $u,v$ or $w=0$ (contradiction).

 In case (iii), if v and 2u are squares  
 $$a^2 = 2m$$
$$ b^2 = 2(1-m^2)$$
Taking products yields:
$$(ab)^2 = 4(m-m^3)$$
 $$(ab/2)^2 = (-m)^3-(-m)$$ 
 So we get a point on the curve 
 $$y^2 = x^3-x$$
 Magma tells us this curve has has rank 0 and $E(K)_{tor} \approx C_2 \oplus C_{4}$  and list of 8 points shows that the x-coordinates of points in $E(K)$ are $0,\pm 1, \pm i$. As above we reach a contradiction.

 Finally in case (iv) if $v$ and $2u$ are squares  
 $$a^2 = 2m$$
$$ b^2 = 2(1+m^2)$$
Taking products yields:
$$(ab)^2 = 4(m+m^3)$$
 $$(ab/2)^2 = m^3+m$$ 
 So we get a point on the curve 
 $$y^2 = x^3+x$$
 Magma tells us this curve has has rank 0 and $E(K)_{tor} \approx C_2 \oplus C_{2}$  and list of 4 points shows that the x-coordinates of points in $E(K)$ are $0, \pm i$. As above we reach a contradiction.
Therefore if $Q$ lifts in some extension then $Q'$ does not (and vice versa).  In particular, as both $Q$ and $Q'$ would lift in an extension L such that $E(L)_{tor} \approx C_8 \oplus C_{8}$ (as any two order 4 points would lift)
 we conclude that there is no quadratic extension L of K such that  $E(L)_{tor} \approx C_8 \oplus C_{8}$.

If $Q$ and $Q''$ both lift in some (possibly different) extension(s), then one of these pairs consists of squares in $K:$ (i) $u,w$ (ii) $u, iv$ (iii) $v,w$ (iv) $v, iv$.  Each of the first three cases contradict \cref{Fermat} and the last case can't hold since $i$ has no square root in K.

Finally if $Q'$ and $Q''$ both lift then one of these pairs consists of squares in K: (i) $w, 2u$ (ii) $w, 2w$ (iii) $2u, iv$ (iv) $2w, iv$.  Cases (i), (ii) have been addressed above.  In case (iii) we obtain $$c^2 = 4im(1+m^2)$$
and multiplying by $-1$ yields
$$(ic/2)^2 = (im)^3-(im)$$
So we get a point on the elliptic curve 
$$ y^2 = x^3-x$$
As noted above, the x-coordinates of points in $E(K)$ are $0, \pm 1, \pm i$
so m satisfies $im=0, \pm 1$ or $\pm i$.  Hence $m =0, \pm 1$, or $\pm i$ which leads to u,v or $w=0$ (a contradiction).

In case (iv) we obtain $$c^2 = 4im(1-m^2)$$

$$(c/2)^2 = (im)^3+(im)$$
So we get a point on the elliptic curve 
$$ y^2 = x^3+x$$
As noted above, the x-coordinates of points in $E(K)$ are $0, \pm i$
so m satisfies $im=0$ or $\pm i$.  Hence $m =0, \pm 1$, which leads to u,v or $w=0$.
We conclude that there is at most one quadratic extension L of K in which $E_{tor}$ grows and in this case $E(L) = C_4 \oplus C_{8}$.\\

\end{proof}

\subsection{$E(K)_{tor} \approx C_2 \oplus C_2$}

\begin{proof}
4)Finally, suppose $E(K) \approx C_2 \oplus C_{2}$.  Fix $d \in K$, d a nonsquare.  We wish to study $E(K(\sqrt{d}))$.  But as E is isomorphic to $E^d$ over $K(\sqrt{d})$,  we have $E(K(\sqrt{d})) \approx E^d(K(\sqrt{d}))$ so it suffices to study the growth of $E^d(K)$ over $K(\sqrt{d})$ \cite{Kwon}*{p.161}.  If $E^d(K)$  is not isomorphic to $C_2 \oplus C_2$ then we may apply a previously proven case.  On the other hand, if $E^d(K) \approx C_2 \oplus C_2$ then it follows that $E(K(\sqrt{d})) \approx C_2 \oplus C_2, C_2 \oplus C_4$ or $C_4 \oplus C_4$.

Now we claim that if $K=\mathbb{Q}(\sqrt{d})$ ($d=-1,3$) and $E(K) \approx C_2 \oplus C_{2}$ then $E(K)_{tor}$ grows in at most 3 quadratic extensions.  Note that E(K)$[2^{\infty}]$ grows in a quadratic extension L if and only if one of the three points (0,0), (-M,0), (-N,0) lifts in L.  But the point (0,0) lifts in at most one quadratic extension, namely K$(\sqrt{M}, \sqrt{N})$ (when it is quadratic over K) by \cref{lifting lemma}.  Similarly, (-M,0) (resp. (-N,0)) can only lift in the  extension K($\sqrt{-M},\sqrt{N-M}$)  (resp. K($\sqrt{-N},\sqrt{M-N}$).  Hence E(K)$[2^{\infty}]$ grows in at most 3 extensions.  On the other hand, by \cref{oddpartsadd}, the number of quadratic extensions where E(K)[$\overline{2}$] grows equals the number of d-quadratic twists (d defined up to a square) with  $E^d(K)$[$\overline{2}$] nontrivial.  But by \cref{twist} there can be at most one such twist (namely $C_2 \oplus C_6$, so E(K)[$\overline{2}$] grows in at most one extension.  

Now if E(K)$[2^{\infty}]$ grows in 3 quadratic extensions then K$(\sqrt{M}, \sqrt{N})$ is quadratic over K, so either (i) $M=ds^2, N=dt^2$ for some $s,t,d \in K$, d a nonsquare or (ii) (WLOG) $M=s^2$ for some $s \in K$ and N is a nonsquare.  If (i) holds, then $E^d(K)$ has a point of order 4 by \cref{lifting lemma} so E cannot have a twist with a point of order 3 by \cref{twist} and hence E(K)[$\overline{2}$] cannot grow in a quadratic extension by \cref{oddpartsadd}. Now if -M is a nonsquare in K (i.e. d $ \not = -1$ modulo squares) then (-M,0) must lift in K($\sqrt{-M}$).  But in this case -N is not a square as well, so (-N,0) must lift in K($\sqrt{-N}$) = K($\sqrt{-M}$) (a contradiction, as we assume the three order two points lift in pairwise different extensions).  Hence it must be the case that -M is a square, so WLOG d=-1.  As this is a nonsquare in K by assumption, this cannot happen over K= Q(i).  On the other hand, over K($\sqrt{-3}$) let E: $y^2=x(x-1)(x-9)$.  Then $E(K)[2^{\infty}]$ grows in K(i), K$(\sqrt{-2})$ and  K$(\sqrt{2})$.

If (ii) holds then as N is a nonsquare in K, (0,0) lifts in K($\sqrt{N}$). But if K$=\mathbb{Q}(i)$, -N is a nonsquare as well so (-N,0) lifts in K($\sqrt{-N}$) = K($\sqrt{N}$) (a contradiction).  If K($\sqrt{-3}$) then -M is not a square so (-M,0) lifts in K($\sqrt{-M}$)=K($\sqrt{-1}$). But then by \cref{lifting lemma} $N-M = \pm  v^2 $ for some v $\in$ K.  But if $N-M = v^2 $ then $M-N = -v^2 $ so (-N,0) must lift in K($\sqrt{M-N}$) = K($\sqrt{-1}$) as well (a contradiction, as we assumed the points lift in different extensions).  Hence $N-M=-v^2$.  As noted in the proof of \cref{twist}, for all nonzero d $\in$ K we have $E^d(M,N) \approx E^d(-M, N-M) = E^d(-a^2,-v^2)$ and by case (i), the latter curve has a no twist with a point of odd order greater than 1.

In conclusion, if $K=\mathbb{Q}(i)$ then $E(K)[2^{\infty}]$ never grows in 3 distinct extensions.  The curve $E(5^2,5 \cdot 32)$ grows in
K$(\sqrt{5})$,K$(\sqrt{10})$,K$(\sqrt{15})$.If $K=\mathbb{Q}(\sqrt{-3})$ then $E(K)[2^{\infty}]$ can grow in 3 different extensions but for such a curve $E(K)[\overline{2}]$ will not grow in any extension.  The curve E(-1,-9) grows in
K$(\sqrt{2})$,K$(\sqrt{-2})$,K$(\sqrt{-1})$.

\end{proof}

\textbf{Acknowledgements}.

We thank Sheldon Kamienny for his helpful discussions and support.




\end{document}